\documentclass{amsart}
\usepackage{amsfonts, amssymb}
\usepackage[width=40em]{geometry}
\usepackage{lipsum}
\usepackage{color}
\usepackage{tikz}

\theoremstyle{plain}
\newtheorem{theorem}{Theorem}[section]

\newtheorem*{theorema*}{Theorem A}
\newtheorem*{theoremb*}{Theorem B}

\newtheorem{corollary}[theorem]{Corollary}

\newtheorem{definition}[theorem]{Definition}
\newtheorem{example}[theorem]{Example}

\newtheorem{lemma}[theorem]{Lemma}

\newtheorem{proposition}[theorem]{Proposition}
\newtheorem{remark}[theorem]{Remark}

\numberwithin{equation}{section}
\DeclareMathOperator{\supp}{supp}
\DeclareMathOperator{\loc}{loc}
\DeclareMathOperator{\R}{\mathbb{R}}
\DeclareMathOperator{\N}{\mathbb{N}}
\newcommand{\dimqa}{\dim\sb{\,\mathrm{qA}}\,}
\newcommand{\dima}{{\dim}\sb{\,\mathrm{A}}\,}

\begin{document}
\title[Quasi-doubling of self-similar measures]{Quasi-doubling of
self-similar measures with overlaps}
\author{Kathryn E. Hare}
\address{Kathryn E. Hare, Department of Pure Mathematics, University of
Waterloo, Canada.}
\email{kehare@uwaterloo.ca}
\author{Kevin G. Hare}
\address{Kevin G. Hare, Department of Pure Mathematics, University of
Waterloo, Canada.}
\email{kghare@uwaterloo.ca}
\author{Sascha Troscheit}
\address{Sascha Troscheit, Department of Pure Mathematics, University of
Waterloo, Canada.}
\email{stroscheit@uwaterloo.ca}
\thanks{KEH was supported by NSERC Grant 2016-03719. KGH was supported by
NSERC Grant 2014-03154. ST was supported by NSERC Grants 2014-03154 and
2016-03719, and the University of Waterloo.}

\begin{abstract}
The Assouad and quasi-Assouad dimensions of a metric space provide
information about the extreme local geometric nature of the set. The Assouad
dimension of a set has a measure theoretic analogue, which is also known as
the upper regularity dimension. One reason for the interest in this notion
is that a measure has finite Assouad dimension if and only if it is doubling.

Motivated by recent progress on both the Assouad dimension of measures that
satisfy a strong separation condition and the quasi-Assouad dimension of
metric spaces, we introduce the notion of the quasi-Assouad dimension of a
measure. As with sets, the quasi-Assouad dimension of a measure is dominated
by its Assouad dimension. It dominates both the quasi-Assouad dimension of
its support and the supremal local dimension of the measure, with strict
inequalities possible in all cases.

Our main focus is on self-similar measures in $\mathbb{R}$ whose support is
an interval and which may have `overlaps'. For measures that satisfy a
weaker condition than the weak separation condition we prove that finite
quasi-Assouad dimension is equivalent to quasi-doubling of the measure, a
strictly less restrictive property than doubling. Further, we exhibit a
large class of such measures for which the quasi-Assouad dimension coincides
with the maximum of the local dimension at the endpoints of the support.
This class includes all regular, equicontractive self-similar measures
satisfying the weak separation condition, such as convolutions of uniform
Cantor measures with integer ratio of dissection. Other properties of this
dimension are also established and many examples are given.
\end{abstract}

\maketitle

\section{Introduction}

The Assouad dimension of a metric space is an indication of its `thickness'
and is of great use in solving embedding problems, see \cite{assouadphd}.
Recently, the Assouad dimension has attracted significant attention in the
metric geometry community, especially when studying dynamical objects such
as attractors and fractals, see for instance \cite{Fraser14, Luukkainen,
Robinson}. It gives quantitative information about the `worst' possible
scaling of a set. The quasi-Assouad dimension was introduced by L\"{u} and
Xi \cite{LuXi} and differs from the Assouad dimension by ignoring some
subexponential effects. It is a lower bound on the Assouad dimension and an
upper bound on the Hausdorff and upper box dimensions of the set. Although
these dimensions will often coincide, such as for self-similar sets
satisfying the open set condition, there are important examples where the dimensions
are different. We refer the reader to \cite{FraserTroscheit18} and \cite%
{Hare} for deterministic and stochastic examples.

As with the Hausdorff dimension, there is an analogue of the Assouad
dimension of the measure. This dimension is also known as the upper
regularity dimension and was first studied by K\"{a}enm\"{a}ki in \cite%
{Kaen1,Kaen2}. As with the Assouad dimension of a set, it captures the worst
scaling behaviour of a measure. One reason the Assouad dimension of a
measure $\mu $ is of interest is because it is finite if and only if the
measure is doubling \cite{FraserHowroyd}, meaning, there is some constant $c$
such that $\mu (B(x,r))\geq c\mu (B(x,2r))$ for every $x\in \supp\mu $.

In this article we introduce and investigate the measure-theoretic analogue
of the quasi-Assouad dimension. We begin the paper by defining the
quasi-Assouad dimension of a measure and proving basic properties. In
particular, we show that the quasi-Assouad dimension is always dominated by
the Assouad dimension of the measure and always dominates all local
dimensions of the measure, with strict inequalities possible in each case.

Following this introductory material, we focus mainly on the quasi-Assouad
dimension of self-similar measures. As remarked in \cite{FraserHowroyd},
\textquotedblleft self-similar measures not satisfying the strong separation
condition are typically not doubling", and thus have infinite Assouad
dimension. In contrast, we show that doubling is not a requirement for
finite quasi-Assouad dimension, thus it is of interest to study the
quasi-Assouad dimension of self-similar measures with `overlap', those which
fail the open set condition. Our paper is primarily concerned with
self-similar measures that have support equal to $[0,1]$ and arise from an
iterated function system that satisfies a natural separation condition, such
as the weak separation condition. One class of examples are the Bernoulli
convolutions with contraction factor the inverse of a Pisot number.

There are two major objectives to this paper. The first is to characterize
finite quasi-Assouad dimension for such measures in terms of a geometric
doubling-like condition that is strictly weaker than doubling. In
particular, as a special case of Theorem \ref{qdimpliesqA}, (see in
particular Corollary~\ref{char}) we obtain

\begin{theorema*}
If $\mu $ is a self-similar measure that satisfies the weak separation
condition and has support $[0,1],$ then the quasi-Assouad dimension of $\mu $
is finite if and only if for every $\varepsilon >0$ there is a constant $%
C_{\varepsilon }$ such that 
\begin{equation}
\mu (B(x,r))\geq C_{\varepsilon }r^{\varepsilon }\mu (B(x,2r))\text{ for all 
}0<r<1\text{ and }x\in \supp\mu .  \label{qdwsc}
\end{equation}
\end{theorema*}

It is easy to see that any measure that is doubling satisfies property (\ref%
{qdwsc}). More generally, we also prove that any measure (whether
self-similar or not) that has finite quasi-Assouad dimension satisfies this
property.

Fraser and Howroyd prove in \cite{FraserHowroyd} that if $\mu $ is a
self-similar measure that satisfies the strong separation condition, then
the Assouad dimension of $\mu $ is equal to the supremal local dimension of $%
\mu $. Hence the quasi-Assouad dimension of the measure is also equal to its
supremal local dimension. It is natural to ask if this is more generally
true for self-similar measures. In fact, it need not be true even for
self-similar measures that satisfy the open set condition since their
quasi-Assouad dimension can be infinite; see Example \ref{exam:OSC}.

The second major goal of the paper is to prove that the quasi-Assouad
dimension is equal to the maximum local dimension for an interesting class
of `overlapping' self-similar measures which we now briefly describe. An
equicontractive self-similar measure is said to be\textit{\ regular} if the
probabilities associated with the left and right-most contractions from the
underlying iterated function system are equal and minimal. Much studied
examples of equicontractive, regular self-similar measures that satisfy the
weak separation condition include Bernoulli convolutions with contraction
ratios the inverse of Pisot numbers and convolutions of uniform Cantor
measures on Cantor sets with integer ratios of dissection; c.f. \cite{F1,
HHM, PSS, Sh}. An extension of the notion of regular to self-similar
measures that are not necessarily equicontractive is known as \textit{%
generalized regular}; see Definition \ref{genreg}. It is a consequence of
Theorem~\ref{thm:theLocDim} (see in particular Corollary~\ref{cor:regqA})
that

\begin{theoremb*}
If $\mu $ is a generalized regular, self-similar measure that satisfies the
weak separation condition and has support $[0,1]$, then the quasi-Assouad
dimension of $\mu $ is 
\begin{equation*}
\dim _{qA}\mu =\max \{\overline{\dim }_{\loc}\,\mu (x):x\in \supp\mu \}.
\end{equation*}%
In particular, this is true for equicontractive, regular self-similar
measures.
\end{theoremb*}

As we observe in Remark \ref{wkcomp}, it is not necessary for a measure to
be generalized regular for its quasi-Assouad dimension to coincide with its
maximum upper local dimension. This is a consequence of the bounds we obtain
in Theorem \ref{LocComp} for the quasi-Assouad dimension of a larger class
of sets known as \textit{weakly comparable}. It is unknown if all the weakly
comparable measures have the property that their quasi-Assouad dimensions
coincide with their maximum upper local dimensions.

The results stated above actually hold for a strictly larger class of
self-similar measures than those satisfying the weak separation condition.
We call this weaker property the \textit{asymptotic gap weak separation
condition}. This property, described in Definition \ref{agwsc}, is closely
related to the asymptotic weak separation condition introduced by Feng in 
\cite{FengSalem} (and may coincide with it). One reason for the interest in
this property is that it is satisfied by Bernoulli convolutions with
contraction ratios the inverse of Salem numbers since these measures do not
satisfy the weak separation condition.\smallskip

The paper is organized as follows. In Section \ref{sect:basicProp} we define
the quasi-Assouad dimension of a measure and prove basic properties such as
the relationship with the Assouad dimension and the local dimensions of the
measure. In Section \ref{sect:IFSs}, we introduce the asymptotic gap weak
separation condition and see that it lies between the weak separation
condition and the asymptotic weak separation condition. Many of our results
on the quasi-Assouad dimension of self-similar measures that satisfy this
separation property rely upon good estimates of the measure of net
intervals, certain subintervals of $[0,1]$ that arise naturally through the
iterative process. Useful technical results relating the measures of net
intervals and the measures of balls can be found in Section \ref{sect:IFSs}.
These were motivated by the study of measures of finite type; see \cite{F3}, 
\cite{HHM}, or \cite{HHS} for further background information on net
intervals, measures of finite type, and other related notions.

In Section \ref{sect:finite} we characterize finite quasi-Assouad dimension
for measures satisfying the asymptotic gap weak separation condition in
terms of the quasi-doubling property and show that all measures with finite
quasi-Assouad dimension have this property. In Sections \ref{sect:wc} and %
\ref{sect:regular} we introduce weakly comparable and generalized regular
measures. Our proof that the generalized regular, self-similar measures
satisfying the asymptotic gap weak separation condition have quasi-Assouad
dimension equal to their maximum upper local dimension is given in Section %
\ref{sect:dimreg}. Dimensional properties of weakly comparable measures are
found in Section \ref{sect:dimwc}. We conclude the paper by giving an
example that demonstrates the importance of the assumption that the support
of the self-similar measure is $[0,1]$.

\section{Basic properties of the quasi-Assouad dimension}

\label{sect:basicProp} Given $X$, a compact subset of $\R^d$, we write $%
N_r(E)$ for the least number of sets of diameter at most $r$ that are
required to cover $E$. Let 
\begin{equation*}
h(\delta)= \inf\left\{ \alpha : (\exists C_1,C_2>0)(\forall
0<r<R^{1+\delta}<R<C_1) \ \sup_{x\in X}N_r(B(x,R)\cap E)\leq C_2 \left( 
\frac{R}{r} \right)^\alpha \right\} .
\end{equation*}
The Assouad dimension of $E$ is given by 
\begin{equation*}
{\dim}\sb{\,\mathrm{A}}\, E=h(0).
\end{equation*}%
The quasi-Assouad dimension is characterized by an exponential gap between $%
r $ and $R$ and is given by 
\begin{equation*}
\dim\sb{\,\mathrm{qA}}\, E=\lim_{\delta\to0}h(\delta).
\end{equation*}

Note that it is not necessary to have both constants, $C_1$ and $C_2$, in
the definition above. However, we introduced both constants as it is
convenient to be able to change between the two definitions.

Our interest in this paper is to study a natural analogue of the Assouad and
quasi-Assouad dimension for measures. By a measure we will always mean a
Borel probability measure on $\R^{d}$.

\begin{definition}
Given a measure $\mu$ and $\delta\geq 0$, set 
\begin{equation*}
H(\delta ) =\inf \left\{s:(\exists C_1,C_2>0)(\forall 0<r\leq R^{1+\delta
}\leq R\leq C_1) \sup_{x\in\supp\mu}\frac{\mu (B(x,R))}{\mu (B(x,r))} \leq
C_2\left( \frac{R}{r}\right)^{s}\right\}.
\end{equation*}
The \textbf{Assouad dimension} of $\mu$ is ${\dim}\sb{\,\mathrm{A}}%
\,\mu=H(0) $ and the \textbf{quasi-Assouad dimension} of $\mu$ is $\dim%
\sb{\,\mathrm{qA}}\,\mu=\lim_{\delta \rightarrow 0}H(\delta)$.
\end{definition}

We note that the limit must exist by monotonicity, but may be infinite. The
Assouad dimension of a measure has also been referred to as the upper
regularity dimension, see \cite{FraserHowroyd} and \cite{Kaen2}.

Since $H(\delta)$ is a non-decreasing function, we clearly have $\dim%
\sb{\,\mathrm{qA}}\,\mu\leq{\dim}\sb{\,\mathrm{A}}\,\mu$. As with the
quasi-Assouad/Assouad dimensions of sets, the quasi-Assouad and Assouad
dimensions of a measure need not coincide and it is even possible for the
Assouad dimension of a measure to be infinite, while the quasi-Assouad
dimension is finite. See Example \ref{strictex}.

In \cite[Theorem 2.1]{FraserHowroyd} it is shown that ${\dim}%
\sb{\,\mathrm{A}}\,\supp\mu \leq {\dim}\sb{\,\mathrm{A}}\,\mu$. The
analogous statement holds for the quasi-Assouad dimension of measures.

\begin{proposition}
Let $\mu$ be a Borel probability measure on $\R^d$ with bounded support.
Then 
\begin{center}
\begin{tikzpicture}
  \node at (-1.1,0) {$\dimqa \supp \mu$};
  \node at (4.3,0) {$\dima\mu$.};
  \node at (2,0.5) {$\dima\supp\mu$};
  \node[rotate=20] at (0.53,0.3) {$\leq$};
  \node[rotate=-20] at (0.54,-0.3) {$\leq$};
  \node at (2,-0.5) {$\dimqa\mu$};
  \node[rotate=20] at (3.33,-0.34) {$\leq$};
  \node[rotate=-20] at (3.33,0.34) {$\leq$};
\end{tikzpicture}
\end{center}
\end{proposition}

\begin{proof}
Suppose that $s=\dim\sb{\,\mathrm{qA}}\,\mu$ and $t=\dim\sb{\,\mathrm{qA}}\,%
\supp\mu$. Then, given $\varepsilon >0$ and $\delta >0$ there exists a
constant $c_1>0$ such that for all $r\leq R^{1+\delta }$ and all $x\in\supp%
\mu$,%
\begin{equation*}
\frac{\mu (B(x,R))}{\mu (B(x,r))}\leq c_1\left( \frac{R}{r}\right)
^{s+\varepsilon }.
\end{equation*}
Similarly, there exists $c_2>0$ and points $y\in \supp\mu$ and $r,R$ with $%
2r\leq R^{1+\delta }$, such that 
\begin{equation*}
N_{2r}(B(y,R)\cap \supp\mu )\geq c_2\left( \frac{R}{2r}\right)
^{t-\varepsilon }.
\end{equation*}

Let $B_{j}(x_{j},r)$, \thinspace $j=1,...,k$ be a maximal collection of
disjoint balls with centres in $B(y,R)\cap \supp\mu$. Then $\bigcup
_{j=1}^{k}B(x_{j},2r)$ covers $B(y,R)\cap \supp \mu$, and so $k\geq
N_{2r}(B(y,R)\cap \supp \mu )$.

These comments imply 
\begin{equation*}
\mu (B(y,2R))\geq k\min_{j}\mu (B(x_{j},r))=k\mu (B(x_{j_{0}},r))
\end{equation*}
for a suitable choice of index $j_{0}$. Moreover, $B(y,2R)\subseteq
B(x_{j_{0}},4R)$, thus 
\begin{align*}
c_2\left( \frac{R}{2r}\right) ^{t-\varepsilon } &\leq N_{2r}(B(y,R)\cap \supp%
\mu )\leq k\leq \frac{\mu (B(y,2R))}{\mu (B(x_{j_{0}},r))} \leq \frac{\mu
(B(x_{j_{0}},4R))}{\mu (B(x_{j_{0}},r))}\leq c_1\left( \frac{2R}{r}\right)
^{s+\varepsilon }\text{.}
\end{align*}
Since we can find arbitrarily small $R$ satisfying this inequality, we must
have $s\geq t$.
\end{proof}

\begin{example}
\label{strictex} (i) A measure $\mu $ satisfying $\dim\sb{\,\mathrm{qA}}\,%
\supp\mu <\dim\sb{\,\mathrm{qA}}\,\mu <{\dim}\sb{\,\mathrm{A}}\,\mu $: Let $%
C $ be the classic middle third Cantor set. We will label the Cantor
intervals at step $n$ of the standard construction as $I_{\omega }$ where $%
\omega \in \{0,1\}^{n}$, with the meaning that if $I_{\nu }$ is a Cantor
interval of step $n-1$, then its left descendent is the interval labelled $%
I_{\nu 0}$ and the right descendent is labelled $I_{\nu 1}$. Choose a sparse
sequence $(n_{k})$. Put $p_{0}^{(n)}=1/3$, $p_{1}^{(n)}=2/3$ if $n\neq n_{k}$
and $p_{0}^{(n_{k})}= 1/4,$ $p_{1}^{(n_{k})}=3/4$. We define the measure $%
\mu $ by the rule that $\mu (I_{\omega })=p_{\omega _{1}}^{(1)}p_{\omega
_{2}}^{(2)}\cdot \cdot \cdot p_{\omega _{n}}^{(n)}$ for $\omega =(\omega
_{1},...,\omega _{n})$. The support of $\mu$ is $C$. Provided $(n_{k})$ is
sufficiently sparse, it can be verified that 
\begin{equation*}
{\dim}\sb{\,\mathrm{A}}\,\supp\mu=\dim\sb{\,\mathrm{qA}}\,\supp\mu =\dim_{%
\mathrm{H}}C=\log 2/\log 3,
\end{equation*}
\begin{equation*}
1=\frac{-\log p_0^{(n)}}{\log 3}=\dim\sb{\,\mathrm{qA}}\,\mu,
\end{equation*}
and 
\begin{equation*}
\log 4/\log 3 = -\log (\liminf p_{0}^{(n_{k})})/\log 3={\dim}%
\sb{\,\mathrm{A}}\,\mu .
\end{equation*}
If instead $p_{0}^{(n_{k})}=1/k$, then ${\dim}\sb{\,\mathrm{A}}\,\mu =\infty 
$. The details are left to the reader. \vskip0.5em (ii) A measure $\mu$
satisfying ${\dim}\sb{\,\mathrm{A}}\,\supp\mu > \dim\sb{\,\mathrm{qA}}\,\mu$%
: By taking as $\mu$ the uniform Cantor measure on a Cantor set $C$ with
suitably varying ratios of dissection, we can arrange for%
\begin{equation*}
\dim\sb{\,\mathrm{qA}}\, \supp\mu=\log2/\log3=\dim\sb{\,\mathrm{qA}}\, \mu <{%
\dim}\sb{\,\mathrm{A}}\, \supp\mu={\dim}\sb{\,\mathrm{A}}\, \mu.
\end{equation*}

(iii) By varying both the probabilities and ratios we can construct measures 
$\mu$ satisfying 
\begin{equation*}
\dim\sb{\,\mathrm{qA}}\,\supp\mu <{\dim}\sb{\,\mathrm{A}}\,\supp\mu <\dim%
\sb{\,\mathrm{qA}}\,\mu<{\dim}\sb{\,\mathrm{A}}\,\mu \text{ or}
\end{equation*}
\begin{equation*}
\dim\sb{\,\mathrm{qA}}\,\supp\mu <\dim\sb{\,\mathrm{qA}}\,\mu<{\dim}%
\sb{\,\mathrm{A}}\,\supp\mu<{\dim}\sb{\,\mathrm{A}}\,\mu.
\end{equation*}
\end{example}

In \cite[Proposition 3.1]{FraserHowroyd} it was observed that the Assouad
dimension of a measure is finite if and only if the measure is doubling. In
Example \ref{NotDoubling} we show that a measure can fail to be doubling,
but have finite quasi-Assouad dimension. In Section~\ref{sect:finite} we
characterize finite quasi-Assouad dimension in terms of a weaker
doubling-type condition.

Recall that the \textit{{\textbf{upper local dimension}}} of a Borel
probability measure $\mu$ at $x\in \supp \mu$ is 
\begin{equation*}
\overline{\dim}_{\loc}\,\mu (x)=\limsup_{r\rightarrow 0}\frac{\log \mu
(B(x,r))}{\log r}.
\end{equation*}%
The \textit{{\textbf{lower local dimension}}}, denoted $\underline{\dim}_{%
\loc}\,\mu (x)$, is defined analogously, replacing $\limsup$ by $\liminf$.
If the upper and lower local dimensions coincide, their common value is
known as the \textit{{\textbf{local dimension}}} of $\mu$ at $x$ and we
write $\dim_{\loc}\,\mu (x)$.

The Assouad dimension is bounded below by the upper local dimension for any
point in the support of $\mu$, see \cite{FraserHowroyd}. The same holds for
the quasi-Assouad dimension.

\begin{proposition}
\label{prop:boundedBylocDim} Let $\mu$ be a Borel probability measure on $\R%
^d$. Then 
\begin{equation}
\dim\sb{\,\mathrm{qA}}\,\mu \geq \sup \{\overline{\dim}_{\loc}\,\mu (x):x\in %
\supp\mu \}.  \label{eq:boundedBylocDim}
\end{equation}
\end{proposition}

\begin{proof}
Let $s=\sup\{\overline{\dim}_{\loc}\mu (x) : x\in\supp\mu\}$. Temporarily
fix $\tau >0$. We will show the quasi-Assouad dimension of $\mu$ is at least 
$s-\tau $.

To begin the proof, fix $\delta>0$, choose $0<\varepsilon <\tau /2$ so small
that $\delta (\tau -2\varepsilon )>3\varepsilon $ and select $x\in \supp \mu$
such that $\overline{\dim}_{\loc}\,\mu (x)\geq s-\varepsilon $. Choose a
decreasing sequence $R_{n}\rightarrow 0$ such that $R_{n+1}\leq
R_{n}^{1+\delta }$ and satisfying the property%
\begin{equation*}
s-2\varepsilon \leq \frac{\log \mu (B(x,R_{n}))}{\log R_{n}}\leq
s+\varepsilon
\end{equation*}%
for all $n.$ Then%
\begin{equation*}
\frac{\mu (B(x,R_{n}))}{\mu (B(x,R_{n+1}))}\geq \frac{R_{n}^{s+\varepsilon }%
}{R_{n+1}^{s-2\varepsilon }}.
\end{equation*}

The choice of $\varepsilon $ ensures that $R_{n+1}^{\tau -2\varepsilon }\leq
R_{n}^{(1+\delta )(\tau -2\varepsilon )}\leq R_{n}^{\tau +\varepsilon }$ and
therefore%
\begin{equation*}
\frac{R_{n}^{s+\varepsilon }}{R_{n+1}^{s-2\varepsilon }}\;\geq\; \left( 
\frac{R_{n}}{R_{n+1}}\right) ^{s-\tau }.
\end{equation*}%
That implies $\dim\sb{\,\mathrm{qA}}\,\mu \geq s-\tau $, as claimed, and as $%
\tau$ was arbitrary the desired conclusion holds.
\end{proof}

Note that the measure of Example \ref{strictex}(i) has its supremal local
dimension occuring at~$0$ and this value coincides with the quasi-Assouad
(but not the Assouad) dimension of the measure. In Theorem \ref%
{thm:theLocDim} we exhibit a class of measures for which this continues to
be true. However, Example \ref{exam:OSC} shows that it is also possible for
the quasi-Assouad dimension to be strictly larger.

\section{Self-similar Iterated Function Systems}

\label{sect:IFSs}

\subsection{Iterated function systems and separation conditions}

Let $\{S_{j}\}$ be a finite family of contractions on $\R$ such that $%
S_j(x)=r_{j}x+d_{j}$, $j=0,...,m-1$, where $r_{j}>0$. Let $%
(p_{0},...,p_{m-1})$ be a non-degenerate probability vector, i.e.\ $p_{j}>0$
and $\sum_{j=0}^{m-1}p_{j}=1$. We refer to the collection $\{S_j\}$ as an 
\textbf{\textit{iterated function system (IFS)}} and the collection of
tuples $\{S_j,p_j\}$ as a \textbf{\textit{weighted iterated function system}}%
.

There exists a unique non-empty compact set satisfying $F=%
\bigcup_{j=0}^{m-1}S_j(F)$, called the attractor, associated with the
collection $S_j$. As all maps are similarities the attractor is known as a
self-similar set. Throughout this article we will assume that the
self-similar set is the unit line $[0,1]$. There is no loss of generality in
assuming $S_{0}(0)=0 $ and $S_{m-1}(1)=1$ as $r_j>0$.

We can similarly define a unique Borel probability measure by assigning
weights to the maps. The resulting measure has support $F$. In fact, this
measure is simply the projection of a Bernoulli measure from the underlying
symbolic dynamics onto $\R$ and is referred to as the self-similar measure
associated with the weighted iterated function system $\{S_j,p_j\}$. More
precisely, the self-similar measure is the unique probability measure $\mu$
satisfying 
\begin{equation*}
\mu (E)=\sum_{j=0}^{m-1}p_{j}\mu (S_{j}^{-1}(E))\text{ for all Borel sets }E.
\end{equation*}%
We let 
\begin{equation*}
\lambda =\min r_{j}
\end{equation*}
and write $\Omega$ for the set of finite words on the alphabet $%
\{0,1,...,m-1\}$. If all $r_j$ are equal, we say that the iterated function
system is \textbf{\textit{equicontractive}}. Given a (finite) word $%
w=(w_{1},...,w_{n})$, we let $w^{-}=(w_{1},...,w_{n-1})$ and denote the
length of $w$ by $\left\vert w\right\vert $. We usually write words by
concatenation, that is $w=w_1w_2w_3\cdots$, and we define 
\begin{eqnarray*}
r_{w} = r_{w_{1}} r_{w_{2}} \cdots r_{w_{n}}\quad\text{and} \quad p_{w} =
p_{w_{1}} p_{w_{2}}\cdots p_{w_{n}}.
\end{eqnarray*}

Commonly, separation conditions are employed to give precise results about
these attractors and measures. If $F$ is the attractor of the IFS $\{S_{i}\}$
and $S_{i}(F)\cap S_{j}(F)=\varnothing $ for all $i\neq j$, we say that the
IFS satisfies the \textbf{\textit{strong separation condition (SSC)}}. If
there is an open set $U$ such that $S_{i}(U)\subseteq U$ for all $i\in
\Lambda $ and $S_{i}(U)\cap S_{j}(U)=\varnothing $ for all $i\neq j$, we say
that IFS $\{S_{i}\}$ satisfies the \textbf{\textit{open set condition (OSC)}}%
. The OSC is a less restrictive condition than the SSC, but in both cases
the Hausdorff and Assouad dimensions coincide for the attractor and their
common value is given by the unique $s$ that satisfies $\sum r_{i}^{s}=1$,
see \cite[Cor. 2.11]{Fraser14}.

Fraser and Howroyd proved that in the case when self-similar measures $\mu $
satisfy the strong separation condition, the Assouad dimension coincides
with the supremal local dimension of $\mu $, \cite[Theorem 2.4]%
{FraserHowroyd}. Since the quasi-Assouad dimension falls between these two
we immediately have the following.

\begin{proposition}
If $\mu $ is a self-similar measure satisfying the strong separation
condition, then 
\begin{equation*}
\dim _{A}\mu =\dim _{qA}\mu =\sup_{x}\{\overline{\dim }_{loc}\mu (x)\}.
\end{equation*}
\end{proposition}

This, however, can fail when relaxing the condition to the open set
condition as the following example shows.

\begin{example}
A self-similar measure satisfying the OSC, with $\dim _{qA}\mu >\sup_{x}\{%
\overline{\dim }_{loc}\mu (x)\}$.\label{exam:OSC}

Consider the IFS $S_{0}(x)=x/2$ and $S_{1}(x)=x/2+1/2,$ with probabilities $%
p_{0}>p_{1}$. Although the self-similar set of the IFS is $[0,1],$ the open
set condition is satisfied with the open set $U=(0,1)$. Temporarily fix $%
\delta >0$. Choose $N$ large and let $k=\lfloor \delta N\rfloor $. Take $x$
to be the midpoint of the interval $S_{0\,1^{[N+k]}}[0,1]$, where $1^{[N+k]}$
is the word consisting of $N+k$ many letters $1$. This interval has left
endpoint $1/2$ and length $2^{-(N+k)}$. Choose $R=2^{-N}$ and $%
r=2^{-(N+k+2)} $, so $r\leq R^{1+\delta }$. Then $B(x,R)$ contains $%
S_{1\,0^{[N]}}[0,1]$, so $\mu (B(x,R))\geq p_{1}p_{0}^{N}$, while $%
B(x,r)\subseteq S_{0\,1^{[N+k]}}[0,1]$ and hence has $\mu $-measure at most $%
p_{0}p_{1}^{N+k} $. Thus if we are to have 
\begin{equation*}
\frac{\mu (B(x,R))}{\mu (B(x,r))}\leq C_{1}\left( \frac{R}{r}\right) ^{s}
\end{equation*}%
for all large $N$, it must be true that $2^{\delta s}\geq
p_{0}p_{1}^{-(1+\delta )}$, in other words, 
\begin{equation*}
s\geq \log 2\left( \frac{1}{\delta }+1\right) \left\vert \log
p_{1}\right\vert -\frac{|\log p_{0}|}{\delta }=\left\vert \log
p_{1}\right\vert +\frac{\left\vert \log p_{1}\right\vert -\left\vert \log
p_{0}\right\vert }{\delta }.
\end{equation*}%
This inequality shows that $H(\delta )\rightarrow \infty $ as $\delta
\rightarrow 0$, hence $\dim\sb{\,\mathrm{qA}}\,\mu $ is infinite.

As this IFS satisfies the open set condition, it is known that 
\begin{equation*}
\{\overline{\dim}_{\loc}\mu (x):x\in \supp\mu \}=\left[\frac{\left\vert \log
p_{0}\right\vert }{\log 2},\frac{\left\vert \log p_{1}\right\vert }{\log 2}%
\right],
\end{equation*}
thus the inequality of Proposition \ref{prop:boundedBylocDim} can be strict.
\end{example}

\vskip1em

In this article we will focus on the quasi-Assouad dimension and will show
that various desirable properties hold under even weaker conditions that we
will now state.

Recall that $\lambda = \min r_j$. Set%
\begin{equation*}
\Lambda_{n}=\{u\in\Omega :r_{u}\leq \lambda ^{n}\text{ and }%
r_{u^{-}}>\lambda ^{n}\}
\end{equation*}
for the set of words that are comparable to $\lambda^n$. In the
equicontractive setting, $\Lambda_{n}$ are simply the words of length $n$.

Lau and Ngai~\cite{LauNgai} studied self-similar IFS under a weaker
separation condition that limits the number of overlapping distinct images.
This so-called weak separation condition subsequently turned out to be the
`correct' separation condition to consider when dealing with the Assouad
dimension of sets. Indeed, the Assouad dimension coincides with the
Hausdorff dimension for self-similar sets satisfying the weak separation
condition and is maximal otherwise, see \cite{FHOR}.

We now recall this definition. Let 
\begin{equation*}
\mathcal{A}(x,r)=\{v\in\Omega \;:\; \lvert S_v([0,1])\rvert\leq r ,\;\;
\lvert S_{v^-}([0,1]) \rvert> r\text{ and }x\in S_v([0,1]) \}
\end{equation*}
and 
\begin{equation*}
\mathcal{M}(x,r) = \{S_v \;:\; v\in\mathcal{A}(x,r)\}.
\end{equation*}
The self-similar IFS $\{S_i\}$ is said to satisfy the \textbf{\textit{weak
separation condition}} if 
\begin{equation*}
\sup_{r\in(0,1)}\sup_{x\in[0,1]} \# \mathcal{M}(x,r)<\infty.
\end{equation*}

Zerner~\cite{Zerner} showed that this is equivalent to the identity not
being an accummulation point of 
\begin{equation*}
\mathcal{E}=\{S_v^{-1}\circ S_w \;:\; v,w\in\Omega\}
\end{equation*}
with respect to the pointwise topology.

Iterated function systems generating Bernoulli convolutions where the
contraction ratio is the reciprocal of a Pisot number, as well as iterated
function systems of the form $(S_{j})$ where $S_{j}(x)=x/d+d_{j}$ where $%
d\in \mathbb{N}$ and $d_{j}\in \mathbb{Q}$ are examples that satisfy the WSC.

Using Zerner's definition, it is straightforward to show that Definition~\ref%
{wsc} below is another equivalent way of stating the weak separation
condition. We have opted to state it in this version as this is the form
that we will use in this article.

\begin{definition}
\label{wsc}An iterated function system $\{S_{j}\}$ satisfies the \textbf{%
weak separation condition (WSC)} if there exists $a>0$ such that if $u, w\in
\Lambda_{n}$ and $S_{u }(0)\neq S_{w}(0)$, then 
\begin{equation*}
\left\vert S_{u}(0)-S_{w}(0)\right\vert \geq a\lambda ^{n}\quad\text{and}%
\quad \left\vert S_{u}(1)-S_{w}(1)\right\vert \geq a\lambda ^{n}.
\end{equation*}
\end{definition}

Motivated by Ngai and Lau's original definition of the weak separation
condition, Feng~\cite{FengSalem} introduced a separation condition, known as
the asymptotic weak separation condition, also in terms of overlapping
images.

\begin{definition}
\label{awsc*} Let $\{S_i\}$ be a self-similar IFS. We say that $\{S_i\}$
satisfies the \textbf{asymptotic weak separation condition (AWSC)} if there
exists non-decreasing function $g(r)$ such that 
\begin{equation*}
\log g(r)/\log r\to 0\quad\text{ and }\quad\sup_{x\in[0,1]}\#\mathcal{M}%
(x,r)\leq g(r).
\end{equation*}
\end{definition}

It is easily observed that the AWSC is weaker than the WSC. In a similar
fashion, we define a useful separation condition on the asymptotic
separation of images of $0$ and $1$.

\begin{definition}
\label{agwsc} An iterated function system $\{S_j\}$ satisfies the \textbf{%
asymptotic gap weak separation condition (AGWSC)} if there exists some
non-increasing function $f(n)>0$ such that $(\log f(n))/n\rightarrow 0$ as $%
n\rightarrow \infty $ and%
\begin{equation*}
\left\vert S_{u}(0)-S_{w}(0)\right\vert \geq f(n)\lambda ^{n} \quad\text{and}%
\quad \left\vert S_{u}(1)-S_{w}(1)\right\vert \geq f(n)\lambda ^{n}
\end{equation*}%
whenever $u,w\in \Lambda_{n}$ and $S_{u}(0)\neq S_{w}(0)$.
\end{definition}

Note that the AGWSC is similar in spirit to the AWSC by allowing the
defining feature to vary on a subexponential scale rather than be finite. We
will show below that the AGWSC implies the AWSC. While closely related, we
are not able to show that these two conditions are equivalent. Note,
however, that the two notions coincide in the only known family to satisfy
the AWSC (or AGWSC), but not the WSC. This family are the IFSs generating
Bernoulli convolutions with contraction ratio the inverse of Salem numbers,
see \cite{FengSalem}.

\begin{lemma}
Let $\{S_i\}$ be a self-similar IFS of the unit line that satisfies the
AGWSC. Then $\{S_i\}$ satisfies the AWSC.
\end{lemma}

\begin{proof}
Let $x_0\in[0,1]$ and let $v\in\Omega$ be such that $x_0\in S_v([0,1])$. Set 
$r= \lvert S_v([0,1]) \rvert$, then any element $S\in\mathcal{M}(x_0, r)$
satisfies $\gamma r<\lvert S([0,1])\rvert \leq r$ for some uniform $\gamma>0 
$. Thus 
\begin{equation*}
S(0),S(1)\in [S_v(0)-r,S_v(0)+2r].
\end{equation*}
Now, since the IFS satisfies the AGWSC, no two distinct maps $S_1, S_2\in%
\mathcal{M}(x_0, r)$ may have both $\lvert S_1(0)-S_2(0)\rvert< \lambda^n
\cdot f(n)$ and $\lvert S_1(1)-S_2(1) \rvert < \lambda^n \cdot f(n)$, where $%
n$ is such that $r\leq \lambda^{n+1}$ and $r>\lambda^{n+2}$. Thus, there are
at most 
\begin{equation*}
\frac{3r}{\lambda^n f(n)} \leq \frac{3}{\lambda f(n)}
\end{equation*}
choices for $S(0)$ and $S(1)$, giving 
\begin{equation*}
\# \mathcal{M}(x_0,r) \leq \left( \frac{3}{\lambda f(n)} \right)^2.
\end{equation*}
Now $v$, and thus $r$, was arbitrary and so the above inequality holds for
all $n$. We obtain 
\begin{align*}
\lim_{r\to0}\frac{ \# \mathcal{M}(x_0,r)}{\lvert\log r\rvert} \leq
\lim_{n\to \infty}\frac{2\log3-2\log (\lambda f(n))}{\lvert\log
\lambda^{n+2}\rvert}=0,
\end{align*}
showing that the AWSC is satisfied.
\end{proof}

Since we will fix a self-similar measure by fixing a weighted iterated
function system, we will also refer to a measure $\mu$ as satisfying the
AGWSC or WSC, where we should say ``the weighted iterated function system
associated with $\mu$''.

In the weak separation case we may assume without loss of generality that
the constant $a$ arising in the definition of the weak separation condition
satisfies $a<\lambda$. Similarly, in the asymptotic gap weak separation case
we can assume that $f(n)<\lambda$ for all $n$.

\begin{definition}
Let $\{S_i\}$ be an equicontractive iterated function system with
contraction ratio $\lambda$. Further, let $c>0$ be such that $c(1-\lambda)$
is the diameter of the associated attractor. The IFS is said to be of 
\textbf{finite type} if there is a finite set $F\subseteq \mathbb{R}$ such
that if $u,w\in \Lambda_{n}$ then either 
\begin{equation*}
\left\vert S_{u}(0)-S_{w}(0)\right\vert >c\lambda ^{n}\quad\text{ or }\quad
\lambda ^{-n}(S_{u}(0)-S_{w}(0))\in F.
\end{equation*}
\end{definition}

The notion of finite type was introduced by Ngai and Wang \cite{NW}. It can
also be defined for IFS that are not equicontractive, but as this is more
technical and not needed in this article, we omit its definition. Examples
of iterated function systems of finite type include Bernoulli convolutions
with contraction ratio the inverse of Pisot numbers. For further
information, the interested reader may peruse \cite{HHS} and the references
cited therein.

We have the following inclusions among these classes, all of which are known
to be proper except for the last: 
\begin{equation*}
\text{OSC}\subset \text{Finite Type}\subset \text{WSC}\subset \text{AGWSC}
\subseteq \text{AWSC}.
\end{equation*}

\subsection{Net intervals}

Recall that the attractor of the IFS is assumed to be $[0,1]$. For each $%
n\in \mathbb{N}$, let $h_{1},...,h_{s_{n}}$ denote the elements of the set $%
\{S_{u}(0),S_{u}(1):u\in \Lambda_{n}\}$, listed in increasing order. Put%
\begin{equation*}
\mathcal{F}_{n}=\{[h_{j},h_{j+1}]:1\leq j\leq s_{n}-1\}.
\end{equation*}%
The elements of $\mathcal{F}_{n}$ are called the\textit{\ net intervals of
level }$n$. Of course, the net intervals of a given level $n\geq 1$ cover $%
[0,1]$. The interval $[0,1]$ will be the (unique) net interval of level $0$.
Since $\left\vert S_{u}(0)-S_{u}(1)\right\vert \leq r_{u }\leq \lambda ^{n}$
when $u\in \Lambda_{n}$, any net interval of level $n$ has length at most $%
\lambda ^{n}$. We denote the length of the net interval $\Delta $ by $%
l(\Delta )$.

Each $x\in [0,1]$ belongs to either one or two net intervals of level $n$.
The point $x$ will belong to two net intervals if and only if $x$ is an
endpoint of a net interval of level $n$. In both cases we refer to the net
interval of level $n$ containing $x$ by $\Delta_{n}(x)$, choosing
arbitrarily when it is not unique. Each net interval $\Delta $ of level $n$
is contained in a unique net interval of level $n-1$ which we refer to as
the parent of $\Delta $.

Given a net interval $\Delta $ of level $n$, let 
\begin{equation}  \label{eq:netIntervals}
P_{n}(\Delta )=\sum_{\substack{ u\in \Lambda_{n}  \\ S_{u }[0,1]\supseteq
\Delta }}p_{u}.
\end{equation}%
Note that $P_{n}(\Delta )\geq \min p_{j}^{n}$. Since $P_n(\Delta)$ is the
sum of all weights of words whose images cover the net interval $\Delta$, we
must have $P_{n}(\Delta_{n}(x))\geq \mu (\Delta_{n}(x))$. As $%
l(\Delta_{n}(x))\leq \lambda ^{n}$, the ball $B(x,\lambda ^{n})$ contains $%
\Delta_{n}(x)$. In particular, if $u\in \Lambda_{n}$ and $x\in S_{u}([0,1])$%
, then $S_{u }([0,1])\subseteq B(x,\lambda ^{n})$. Thus, 
\begin{equation}
\mu (B(x,\lambda ^{n}))\geq P_{n}(\Delta_{n}(x))\geq \mu (\Delta_{n}(x)).
\label{PnLower}
\end{equation}

To compare the minimal and maximal contraction rate we define $\Theta\in\N$
implicitly as the least integer satisfying 
\begin{equation}  \label{eq:ThetaDefinition}
\left( \max_{j\in\Lambda}r_j \right)^{\Theta+1}<\lambda^2 = \left(
\min_{j\in\Lambda}r_j \right)^2.
\end{equation}

We collect some further properties of $P_{n}(\Delta )$ below.

\begin{lemma}
\label{PnRel}Let $\Theta$ be as above. Then

\begin{enumerate}
\item[(a)] $P_{n}(\Delta_{n}(x))\geq \min
p_{j}^{\Theta}P_{n-1}(\Delta_{n-1}(x))$ and

\item[(b)] $P_{n}(\Delta_{n}(x))\leq P_{n-1}(\Delta_{n-1}(x))$.
\end{enumerate}
\end{lemma}

\begin{proof}
(a) Let $u\in \Lambda_{n-1}$ and suppose $S_{u}[0,1]\supseteq
\Delta_{n-1}(x) $. Then there exists some word $w$ such that $u w\in
\Lambda_{n}$ and $S_{uw}[0,1]\supseteq \Delta _{n}(x) $. Note that $%
r_{w^{-}}>\lambda$ and so $r_{w}\geq \lambda ^{2}$. Using the definition of $%
\Theta$ we find that $\left\vert w\right\vert \leq \Theta$. As $p_{u}=p_{u
w}p_{w}^{-1}$, we have 
\begin{equation*}
P_{n-1}(\Delta_{n-1}(x))\leq P_{n}(\Delta_{n}(x))\left( \min p_{w }\right)
^{-1}\leq P_{n}(\Delta_{n}(x))(\min p_{j}^{\Theta})^{-1}\text{.}
\end{equation*}

(b) Suppose $v\in \Lambda_{n}$ and $S_{v}[0,1]\supseteq \Delta _{n}(x) $.
Then $v=u w$ where $u \in \Lambda_{n-1}$ and $S_{u }[0,1]\supseteq
\Delta_{n-1}(x)$. Furthermore, the sum of $p_{w}$ taken over such $w$ is at
most one. Thus 
\begin{equation*}
P_{n}(\Delta_{n}(x))=\sum_{\substack{ u w\in \Lambda_{n},u \in \Lambda_{n-1} 
\\ S_{u w}[0,1]\supseteq \Delta_{n}(x)}}p_{u }p_{w}\leq
P_{n-1}(\Delta_{n-1}(x)).\qedhere
\end{equation*}
\end{proof}

Now suppose the iterated function system satisfies the asymptotic gap weak
separation condition with function $f(n)$. For each $n$, choose the minimal
integer $\kappa_{n}$ such that $\lambda ^{\kappa_{n}}\leq f(n)$. By
definition, 
\begin{equation*}
\kappa_{n}\leq \frac{\log f(n)}{\log \lambda }+1
\end{equation*}%
and therefore $\kappa_{n}/n\rightarrow 0$. As $\lambda ^{n+\kappa_{n}}\leq
f(n)\lambda ^{n}$, 
\begin{equation*}
B(x,\lambda ^{n+\kappa_{n}})\subseteq B(x,f(n)\lambda ^{n}).
\end{equation*}
If $u \in \Lambda _{n+\kappa_{n}}$, then $r_{u }\leq \lambda ^{n+\kappa_{n}}$%
. If also $S_{u }[0,1]\supseteq \Delta_{n+\kappa _{n}}(x)$, then $S_{u
}[0,1]\subseteq B(x,f(n)\lambda ^{n})$, and so by \eqref{eq:netIntervals}, 
\begin{equation}
\mu (B(x,f(n)\lambda ^{n}))\geq P_{n+\kappa_{n}}(\Delta_{n+\kappa _{n}}(x)).
\label{PnUpper}
\end{equation}%
If $u \in \Lambda_{n}$ and $S_{u }(0)=0$, then $S_{u}(1)=r_{u}\geq \lambda
^{n+1}\geq f(n)\lambda _{n}$, hence $l(\Delta_{n}(0))\geq f(n)\lambda ^{n}$.
Consequently, 
\begin{equation}
B(0,f(n)\lambda ^{n})\cap \lbrack 0,1]\subseteq \Delta_{n}(0)\subseteq
B(0,\lambda ^{n}).  \label{Delta0}
\end{equation}

Combining these observations with the previous lemma gives the following
bounds.

\begin{proposition}
\label{prop:theAbound} \label{bounds}Suppose $\mu$ is a self-similar measure
that satisfies the asymptotic gap weak separation condition with function $%
f(n)$. There is a constant $0<A<1$ such that for any $x,n$%
\begin{equation*}
A^{\kappa_{n}}P_{n}(\Delta_{n}(x))\leq \mu (B(x,f(n)\lambda ^{n})).
\end{equation*}%
In particular, 
\begin{equation*}
A^{\kappa_{n}}P_{n}(\Delta_{n}(0))\leq \mu (B(0,f(n)\lambda ^{n}))\leq \mu
(\Delta_{n}(0))\leq P_{n}(\Delta_{n}(0)).
\end{equation*}
\end{proposition}

\begin{proof}
Lemma \ref{PnRel}(a) implies $P_{n+\kappa_{n}}(\Delta_{n+\kappa
_{n}}(x))\geq A^{\kappa_{n}}P_{n}(\Delta_{n}(x))$ for $A=\min p_{j}^{\Theta}$%
. This fact, coupled with (\ref{PnUpper}), gives the first statement. The
second statement follows similarly from (\ref{PnLower}) and (\ref{Delta0}).
\end{proof}

Let $s,t$ be 
\begin{equation}
s=\liminf \left( P_{n}(\Delta_{n}(0)\right) ^{1/n}\quad\text{and}\quad
t=\liminf \left( P_{n}(\Delta_{n}(1)\right) ^{1/n}.  \label{s,t}
\end{equation}%
Note that $s,t>0$ since $P_{n}(\Delta_{n}(x))\geq \left(
\min_{j}p_{j}\right) ^{n}$ for all $x,n.$

\begin{proposition}
\label{LocDim0}Suppose $\mu$ is a self-similar measure that satisfies the
asymptotic gap weak separation condition. Then 
\begin{equation*}
\overline{\dim}_{\loc}\,\mu (0)=\frac{\log s}{\log \lambda }\quad\text{ and }%
\quad \overline{\dim}_{\loc}\,\mu (1)=\frac{\log t}{\log \lambda }.
\end{equation*}
Furthermore, $s,t<1$.
\end{proposition}

\begin{proof}
Our earlier observations show that 
\begin{eqnarray*}
B(0,\lambda_{n}^{n+\kappa_{n}})\cap \lbrack 0,1] &\subseteq &B(0,f(n)\lambda
^{n})\cap \lbrack 0,1]\quad \subseteq\quad \Delta_{n}(0)\quad\subseteq\quad
B(0,\lambda ^{n}) \\
&\subseteq &B(0,f(n-\kappa_{n})\lambda ^{n-\kappa_{n}})\cap \lbrack
0,1]\quad\subseteq\quad \Delta_{n-\kappa_{n}}(0),
\end{eqnarray*}%
and by Proposition~\ref{prop:theAbound}, 
\begin{eqnarray*}
A^{\kappa_{n}}P_{n}(\Delta_{n}(0)) &\leq &\mu (B(0,\lambda
^{n}))\quad\leq\quad \mu (\Delta_{n-\kappa_{n}}(0)) \\
&\leq &P_{n-\kappa_{n}}(\Delta_{n-\kappa_{n}}(0))\quad\leq\quad A^{-\kappa
_{n}}P_{n}(\Delta_{n}(0)).
\end{eqnarray*}

Since $\kappa_{n}/n\rightarrow 0$ as $n\to\infty$, 
\begin{align*}
\limsup_{n\to\infty}\frac{\log(P_n(\Delta_n(0)))}{n\log\lambda} \leq&
\limsup_{n\to\infty}\frac{-\kappa_n \log A+\log (P_{n}(\Delta_{n}(0)))}{%
n\log \lambda } \leq \limsup_{n\to\infty}\frac{\log (\mu (B(0,\lambda ^{n})))%
}{n\log \lambda } \\
=\; \overline{\dim}_{\loc}\,\mu(0)\; \leq& \limsup_{n\to\infty}\frac{%
\kappa_n\log A+\log (P_{n}(\Delta_{n}(0)))}{n\log \lambda } \leq
\limsup_{n\to\infty}\frac{\log(P_n(\Delta_n(0)))}{n\log\lambda}.
\end{align*}%
This proves the first equality. The second equality follows similarly, and
is omitted for brevity.

To prove that $s<1$, note that there must be at least one index $j$ such
that $S_{j}(0)\neq 0$. Without loss of generality assume the index is $j_{0}$%
. Thus, if $\tau\in \Lambda_{n}$ with $S_{\tau}[0,1]\supseteq \Delta _{n}(0)$%
, and $\tau=u w$ where $u\in \Lambda_{n-1}$ and $S_{u}[0,1]\supseteq
\Delta_{n-1}(0)$, then $w$ does not contain the letter $j_{0}$. It follows
that 
\begin{equation*}
P_{n}(\Delta_{n}(0))\leq (1-p_{j_{0}})\hspace{-1.3em}\sum_{\substack{ u \in
\Lambda _{n-1}  \\ S_{u}[0,1]\supseteq \Delta_{n-1}}}\hspace{-1.2em}p_{u
}=(1-p_{j_{0}})P_{n-1}(\Delta_{n-1}(0))
\end{equation*}%
and hence $P_{n}(\Delta_{n}(0))\leq (1-p_{j_{0}})^{n}$. Thus $s\leq
1-p_{j_{0}}<1$. The case for $t<1$ follows along the same lines and is
omitted.
\end{proof}

\section{Measures with finite quasi-Assouad dimension}

\label{sect:finite}

In \cite[Proposition 3.1]{FraserHowroyd} it was proven that the Assouad
dimension of a measure is finite if and only if the measure is doubling.
This is not required for finite quasi-Assouad dimension, as is shown in
Example~\ref{NotDoubling}.

In this section we characterize the measures with finite quasi-Assouad
dimension that satisfy the asymptotic gap weak separation condition in terms
of a weak doubling-like condition. This characterization is stated in
Corollary \ref{char}.

\begin{definition}
A compactly supported Borel probability measure $\mu $ is \textbf{%
quasi-doubling} if for every $\varepsilon >0$ and positive, non-decreasing
function $G(r)$ satisfying $\log G(r)/\log r\rightarrow 0$ as $r\rightarrow
0 $, there exists $c>0$ such that%
\begin{equation}
\mu (B(x,G(r)r))\geq c\,r^{\varepsilon }\mu (B(x,2r))  \label{qdoublingdef}
\end{equation}%
for all $x\in \supp\mu $ and $r\in (0,1).$
\end{definition}

By considering $r=\lambda ^{n}$ for any (fixed) $\lambda <1$ and $%
g(n)=G(\lambda ^{n}),$ one can check that this definition is equivalent to
the statement that $\mu $ is quasi-doubling if and only if for every $q>1,$
constant $b>0$ and positive, non-increasing sequence $(g(n))_{n}$ satisfying 
$\log g(n)/n\rightarrow 0$, there exists $c>0$ such that 
\begin{equation}
\mu (B(x,g(n)\lambda ^{n}))\geq c\,q^{-n}\mu (B(x,b\lambda ^{n})),
\label{finiteqA}
\end{equation}%
for all $x\in \supp\mu $ and $n\in \N$. We will often use the definition of
quasi-doubling in this formulation, or for a particular choice of $b,$ which
is also equivalent.

We first check that doubling implies quasi-doubling and that quasi-doubling
is necessary for finite quasi-Assouad dimension.

\begin{proposition}
\label{doubling}(i) If $\mu $ is a doubling measure, then $\mu $ is
quasi-doubling.

(ii) If $\dim _{qA}\mu <\infty ,$ then $\mu $ is quasi-doubling.
\end{proposition}

\begin{proof}
(i) Pick $\lambda <1$ assume the constant $C$ is chosen so that $\mu
(B(x,\lambda r))\geq C\mu (B(x,r))$ for all $x,r$. Fix $q>1$ and choose $%
\varepsilon >0$ so that $C^{\varepsilon }>q^{-1}$. Let $(g(n))_{n}$ be a
non-increasing sequence satisfying $\log g(n)/n\rightarrow 0$. As $\log
g(n)/n\rightarrow 0$ there is some $n_{0}$ such that $g(n)\geq \lambda
^{\varepsilon n}$ for all $n\geq n_{0}$. Since $g$ is non-increasing, $%
g(n)\geq g(n_{0})\geq \lambda ^{\varepsilon n_{0}}\geq A\lambda
^{\varepsilon n}$ for all $n\leq n_{0}$ and a suitable constant $A$. By
repeated application of the doubling property, 
\begin{equation*}
\mu (B(x,g(n)\lambda ^{n}))\geq \mu (B(x,A\lambda ^{\varepsilon n}\lambda
^{n}))\geq C^{n\varepsilon +1}\mu (B(x,A\lambda ^{n}))\geq Cq^{-n}\mu
(B(x,A\lambda ^{n}))
\end{equation*}%
and this gives (\ref{finiteqA}).

(ii) Suppose $\dim \sb{\,\mathrm{qA}}\,\mu <t<\infty $ and let $q>1>\lambda $%
. Select $\delta >0$ so that $\lambda ^{-\delta t}=q$. The definition of the
quasi-Assouad dimension ensures that for some $C>0$, 
\begin{equation}
\frac{\mu (B(x,\lambda ^{n}))}{\mu (B(x,\lambda ^{n}{}^{(1+\delta )}))}\leq
C\lambda ^{-n\delta t}=Cq^{n}.
\end{equation}%
Assume $\log g(n)/n\rightarrow 0$. Then $g(n)\geq \lambda ^{\delta n}$ for $%
n $ sufficiently large. Hence $B(x,\lambda ^{n}{}^{(1+\delta )})\subseteq
B(x,g(n)\lambda ^{n})$ and therefore 
\begin{equation*}
C^{-1}q^{-n}\,\mu (B(x,\lambda ^{n}))\leq \mu (B(x,\lambda ^{n}{}^{(1+\delta
)}))\leq \mu (B(x,g(n)\lambda ^{n}))
\end{equation*}%
for sufficiently large $n$ and this suffices to show $\mu $ is
quasi-doubling.
\end{proof}

Before stating our main result of this section, we introduce further
notation for self-similar measures $\mu $ with support $[0,1]$ and minimal
contraction factor $\lambda $.

Given a level $n$ net interval, $\Delta _{n}$, other than $\Delta _{n}(1)$
or $\Delta _{n}(0)$, we let $\Delta _{n}^{R}$ be the union of the two net
intervals of level $n$ immediately to the right of $\Delta _{n}$, and let $%
\Delta _{n}^{L}$ be the union of the two net intervals immediately to the
left, with the understanding that $\Delta _{n}^{R}=\Delta _{n}(1)$ if $%
\Delta _{n}$ is immediately adjacent to $\Delta _{n}(1)$ and similarly for $%
\Delta _{n}^{L}$. If $\Delta _{n}=\Delta _{n}(1)$ we only define $\Delta
_{n}^{L}$ and if $\Delta _{n}=\Delta _{n}(0)$ we only define $\Delta
_{n}^{R} $.

We remark that if the IFS associated with $\mu $ satisfies the asymptotic
gap weak separation condition with function $f(n)$, any net interval whose
endpoints are $S_{u}(0)$ and $S_{w}(0)$ for some $u,w\in \Lambda _{n}$ has
length at least $f(n)\lambda ^{n}$ and likewise if the endpoints are both
iterates of $1$. Consequently, the length of the union of any two adjacent
net intervals has length at least $f(n)\lambda ^{n}$. In particular, this is
true for $\Delta _{n}^{R}$ and $\Delta _{n}^{L}$ since $\Delta _{n}(0)$ and $%
\Delta _{n}(1)$ also have length at least $f(n)\lambda ^{n}$.

\begin{definition}
\label{def:netDoubl} A self-similar measure $\mu $ that satisfies the
asymptotic gap weak separation condition with function $f(n)$ is \textbf{%
quasi-net doubling} if for every $q>1$ there exist $c_{1},c_{2}>0$ such that
if $l(\Delta _{n})\geq f(n+1)\lambda ^{n+1}$, then 
\begin{equation}
c_{1}q^{-n}\mu (\Delta _{n}^{\ast })\leq \mu (\Delta _{n})\leq c_{2}q^{n}\mu
(\Delta _{n}^{\ast })  \label{qdoubling}
\end{equation}%
for $\Delta _{n}^{\ast }=\Delta _{n}^{L}$ and $\Delta _{n}^{\ast }=\Delta
_{n}^{R}$, where defined.
\end{definition}

\begin{theorem}
\label{qdimpliesqA}Suppose the self-similar measure $\mu$ has support $[0,1]$
and satisfies the asymptotic gap weak separation condition. Then $\dim%
\sb{\,\mathrm{qA}}\,\mu<\infty$ if and only if $\mu$ is quasi-net doubling.
\end{theorem}

\begin{proof}
Suppose $\mu $ satisfies the AGWSC with function $f(n)$ and has finite
quasi-Assouad dimension. By Prop. \ref{doubling}(ii) $\mu $ is
quasi-doubling and, in particular, satisfies (\ref{finiteqA}) for all $q>1$
with $g(n)=f(n+1)\lambda /2$ and $b=3$. We will see that this already
implies $\mu $ is quasi-net doubling. In other words, we will prove that if
there is a constant $c$ so that 
\begin{equation}
\mu (B(x,f(n+1)\lambda ^{n+1}/2))\geq c\,q^{-n}\mu (B(x,3\lambda ^{n}))
\label{qndwsc}
\end{equation}%
for all $x\in \supp\mu $ and positive integers $n,$ then $\mu $ is quasi-net
doubling.

Assume $\Delta _{n}$ is a level $n$ net interval with $l(\Delta _{n})\geq
f(n+1)\lambda ^{n+1}$. Let $z$ be the midpoint of $\Delta _{n}$ and let $%
\Delta _{n}^{\ast }$ refer to either $\Delta _{n}^{L}$ or $\Delta _{n}^{R}$.
Then $B(z,f(n+1)\lambda ^{n+1}/2)\subseteq \Delta _{n}$ and $B(z,3\lambda
^{n})\supseteq \Delta _{n}^{\ast }$. Thus, using \eqref{finiteqA}, 
\begin{equation*}
\mu (\Delta _{n})\,\geq \,\mu (B(z,f(n+1)\lambda ^{n+1}/2))\,\geq
\,c\,q^{-n}\mu (B(z,3\lambda ^{n}))\geq c\,q^{-n}\mu (\Delta _{n}^{\ast }).
\end{equation*}%
This proves the left hand inequality in (\ref{qdoubling}).

Similarly, we can prove the other inequality. Recall that $l(\Delta
_{n}^{\ast })\geq f(n)\lambda ^{n}$ and so, letting $z$ be the midpoint of $%
\Delta _{n}^{\ast }$, and using \eqref{finiteqA}, we obtain 
\begin{equation*}
\mu (\Delta _{n}^{\ast })\geq \mu (B(z,f(n)\lambda ^{n}))\geq cq^{-n}\mu
(B(z,3\lambda ^{n}))\geq cq^{-n}\mu (\Delta _{n}).
\end{equation*}%
This proves the right hand inequality and therefore $\mu $ is quasi-net
doubling.

Now assume $\mu $ is quasi-net doubling and denote by $\lambda $ the minimal
contraction factor. Without loss of generality we can assume $\lambda <1/2$,
for if not, we can replace the IFS $\{S_{j}\}$ with suitable $k$-fold
compositions of the maps $S_{j}$.

Fix $\delta >0$ and let $N_0$ be large enough such that 
\begin{equation}  \label{eq:fNlambdabound}
f(N+1)\geq \lambda^{N\delta/2}
\end{equation}
for all $N\geq N_0$. Such an $N_0$ exists as the asymptotic gap weak
separation condition guarantees $\log(f(n))/n\to0$ as $n\to0$. We will be
using the bounds of \eqref{qdoubling} with $q=2^{\delta /(1+\delta )}>1 $
and consider $\Delta_N(x)$. \vskip1em \textbf{Case 1:} Assume that $%
l(\Delta_{N}(x))\geq f(N+1)\lambda ^{N+1}$. As $\Delta _{N}^{R}$ and $%
\Delta_{N}^{L}$ have length at least $f(N)\lambda ^{N}$, 
\begin{equation*}
B(x,f(N)(1-\lambda )\lambda ^{N})\cap \lbrack 0,1]\subseteq \Delta
_{N}(x)\cup \Delta_{N}^{R}\cup \Delta_{N}^{L}
\end{equation*}%
and thus by the quasi-net doubling condition,%
\begin{eqnarray*}
\mu (B(x,f(N)(1-\lambda )\lambda ^{N})) &\leq &\mu (\Delta_{N}(x))+\mu
(\Delta_{N}^{R})+\mu (\Delta_{N}^{L}) \\
&\leq &c\,q^{N}\mu (\Delta_{N}(x))
\end{eqnarray*}
for some $c>0$. From (\ref{PnLower}), we have $\mu (\Delta_{N}(x))\leq
P_{N}(\Delta _{N}(x))$ and $\mu (B(x,\lambda ^{n}))\geq P_{n}(\Delta_{n}(x))$
for any $n $. Let $t=(\min p_{j})^{-\Theta}$, where $\Theta$ is given in %
\eqref{eq:ThetaDefinition}. It now follows from Lemma \ref{PnRel} that%
\begin{equation*}
\frac{\mu (B(x,f(N)(1-\lambda )\lambda ^{N}))}{\mu (B(x,3\lambda ^{n}))}\leq 
\frac{c\,q^{N}P_{N}(\Delta_{N}(x))}{P_{n}(\Delta_{n}(x))}\leq
c\,q^{N}t^{n-N}.
\end{equation*}

The ``gap'' between $r$ and $R$ in the definition of the quasi-Assouad
dimension means that we can restrict our attention to the case where 
\begin{equation*}
\lambda ^{n}\leq (f(N)(1-\lambda )\lambda ^{N})^{1+\delta }.
\end{equation*}
We can therefore assume without loss of generality that $n\geq
N(1+\delta^{\prime })$ for all $\delta^{\prime }<\delta$. In particular,
this holds for $\delta^{\prime }=\delta/(1+\delta)$. Rearranging gives $%
N\delta/(1+\delta)= N\delta^{\prime }\leq(n-N)$ and hence $q^{N}=
2^{N\delta^{\prime }}\leq 2^{n-N}$. Taking $\beta =\log 2t/|\log \lambda |$
we have 
\begin{equation*}
\frac{\mu (B(x,f(N)(1-\lambda )\lambda ^{N}))}{\mu (B(x,3\lambda ^{n}))}\leq
c\lambda ^{-\beta (n-N)}.
\end{equation*}

Using \eqref{eq:fNlambdabound}, we have 
\begin{equation*}
\frac{f(N+1)(1-\lambda )\lambda ^{N}}{3\lambda ^{n}}\geq c_1\lambda
^{N-n}\lambda ^{N\delta /2}\geq c_1\lambda ^{-(n-N)/2},
\end{equation*}
for all $n\geq(1+\delta^{\prime })N>N\geq N_0$ and some $c_1>0$. Redefining $%
c$, if necessary, we obtain 
\begin{equation}
\frac{\mu (B(x,f(N)(1-\lambda )\lambda ^{N}))}{\mu (B(x,3\lambda ^{n}))}\leq
c\left( \frac{f(N+1)(1-\lambda )\lambda ^{N}}{3\lambda ^{n}}\right)
^{2\beta}.  \label{1}
\end{equation}

\vskip1em We next show that in the second case we obtain the same bound,
before establishing that this is sufficient to guarantee finite
quasi-Assouad dimension.

\vskip1em

\textbf{Case 2:} Assume $l(\Delta_{N}(x))<f(N+1)\lambda ^{N+1}$. In this
case, $\Delta _{N}(x)$ cannot contain two net subintervals of level $N+1$ as
their union would have length at least $f(N+1)\lambda ^{N+1}$. Thus $\Delta
_{N+1}(x)=\Delta_{N}(x)$. Fix $n$ such that $3\lambda ^{n}\leq
(f(N)(1-\lambda )\lambda ^{N})^{1+\delta }$ and choose the maximal integer $%
j $ such that $N<j\leq n$ and $\Delta_{N}(x)= \cdots =\Delta _{j}(x)$.

Since the union of two adjacent level $N$ net intervals has length at least $%
f(N)\lambda ^{N}$, it follows that the level $N$ net intervals immediately
adjacent to $\Delta_{N}(x)$ have length at least $f(N)(1-\lambda )\lambda
^{N}$. Denote the left and right net intervals of $\Delta_N(x)$ by $\Delta
_{N}^{r}$ and $\Delta_{N}^{l}$ respectively. Thus 
\begin{equation*}
B(x,f(N)(1-\lambda )\lambda ^{N})\cap \lbrack 0,1]\subseteq \Delta
_{N}(x)\cup \Delta_{N}^{r}\cup \Delta_{N}^{l}.
\end{equation*}%
Let $x_{1},x_{2}$ be the midpoints of $\Delta_{j}^{r}$ and $\Delta_{j}^{l}$
respectively. As each level $j$ net interval has length at most $\lambda
^{j} $ we have $B(x_{i},\lambda ^{j})\subseteq B(x,3\lambda ^{j})$ for $%
i=1,2 $.

These observations yield the bounds 
\begin{eqnarray*}
\mu (B(x,3\lambda ^{j})) &\geq &\max \left( \mu (B(x,\lambda ^{j})),\mu
(B(x_{1},\lambda ^{j})),\mu (B(x_{2},\lambda ^{j}))\right) \\
&\geq &\max (P_{j}(\Delta_{j}(x)),P_{j}(\Delta_{j}^{r}),P_{j}(\Delta
_{j}^{l}))
\end{eqnarray*}%
and 
\begin{eqnarray*}
\mu (B(x,f(N)(1-\lambda )\lambda ^{N})) &\leq &3\max (\mu (\Delta
_{N}(x)),\mu (\Delta_{N}^{r}),\mu (\Delta_{N}^{l})) \\
&\leq &3\max (P_{N}(\Delta_{N}(x)),P_{N}(\Delta_{N}^{r}),P_{N}(\Delta
_{N}^{l})).
\end{eqnarray*}%
Since $\Delta_{N}(x)=\Delta_{j}(x)$, it follows that $\Delta
_{j}^{r}\subseteq \Delta_{N}^{r}$ and $\Delta_{j}^{l}\subseteq \Delta
_{N}^{l}$, so 
\begin{equation*}
\frac{\mu (B(x,f(N)(1-\lambda )\lambda ^{N}))}{\mu (B(x,3\lambda ^{j}))}\leq
3\max \left( \frac{P_{N}(\Delta_{N}(x))}{P_{j}(\Delta_{j}(x))},\frac{%
P_{N}(\Delta_{N}^{r})}{P_{j}(\Delta_{j}^{r})},\frac{P_{N}(\Delta_{N}^{l}))}{%
P_{j}(\Delta_{j}^{l}))}\right) \leq 3t^{j-N},
\end{equation*}
where $t=(\min p_j)^{-\Theta}$. If $j=n$, then as in Case 1, we have 
\begin{equation}
\frac{\mu (B(x,f(N)(1-\lambda )\lambda ^{N}))}{\mu (B(x,3\lambda ^{n}))}\leq
c\left( \frac{f(N+1)(1-\lambda )\lambda ^{N}}{3\lambda ^{n}}\right)
^{2\beta}.  \label{2}
\end{equation}

Otherwise, $j<n$ so that $\Delta_{j+1}(x)\neq \Delta_{j}(x)$. That ensures $%
\Delta_{j}(x)$ contains at least two $(j+1)$-level net intervals and so its
length is at least $f(j+1)\lambda ^{j+1}$. Thus the quasi-net doubling
condition implies 
\begin{equation*}
\mu (\Delta_{j}(x))+\mu (\Delta_{j}^{R})+\mu (\Delta_{j}^{L})\leq
c\,q^{j}\mu (\Delta_{j}(x))\leq c\,q^{j}P_{j}(\Delta_{j}(x))
\end{equation*}%
and hence%
\begin{equation*}
\frac{\mu (\Delta_{j}(x)\cup \Delta_{j}^{R}\cup \Delta_{j}^{L})}{\mu
(B(x,3\lambda ^{n}))}\leq c\,q^{j}\frac{P_{j}(\Delta_{j}(x))}{P_{n}(\Delta
_{n}(x))}\leq c\,q^{j}t^{n-j}.
\end{equation*}
\vskip1em We will deal with the case where the maximum of $P_{N}(\Delta
_{N}(x)),P_{N}(\Delta_{N}^{r})$, and $P_{N}(\Delta_{N}^{l})$ is $%
P_{N}(\Delta _{N}^{r})$. The other two cases are analogous and left to the
reader.

Let $y_{1}$ be the right endpoint of $\Delta_{j}(x)$. Then 
\begin{equation*}
B(y_{1},f(j+1)\lambda ^{j+1})\cap \lbrack 0,1]\subseteq \Delta_{j}(x)\cup
\Delta_{j}^{R}\cup \Delta_{j}^{L},
\end{equation*}%
so applying (\ref{PnUpper}) we have 
\begin{equation*}
\mu (\Delta_{j}(x))+\mu (\Delta_{j}^{R})+\mu (\Delta_{j}^{L})\geq \mu
(B(y_{1},f(j+1)\lambda ^{j+1})\geq P_{j+1+\kappa_{j+1}}(\Delta_{j+1+\kappa
_{j+1}}(y_{1})).
\end{equation*}%
where we are free to take $\Delta_{j+1+\kappa_{j+1}}(y_1)$ to be the net
interval having $y_1$ as the left endpoint. But as $\Delta_{j}(x)=%
\Delta_{N}(x)$, $y_{1}$ is also the left endpoint of $\Delta_{N}^{r}$. So we
can choose $\Delta_{N}(y_{1})=\Delta_{N}^{r}$. Combining these observations
gives 
\begin{eqnarray*}
\frac{\mu (B(x,f(N)(1-\lambda )\lambda ^{N}))}{\mu (\Delta_{j}(x)\cup
\Delta_{j}^{R}\cup \Delta_{j}^{L})} &\leq &\frac{3P_{N}(\Delta_{N}^{r})}{%
P_{j+1+\kappa_{j}}(\Delta_{j+1+\kappa_{j+1}}(y_{1}))} \\
&\leq &\frac{3P_{N}(\Delta_{N}(y_{1}))}{P_{j+1+\kappa_{j+1}}(\Delta
_{j+1+\kappa_{j+1}}(y_{1}))}\leq ct^{j+\kappa_{j+1}-N}\text{.}
\end{eqnarray*}

Consequently, 
\begin{eqnarray*}
\frac{\mu (B(x,f(N)(1-\lambda )\lambda ^{N}))}{\mu (B(x,3\lambda ^{n}))}
&\leq &\frac{\mu (B(x,f(N)(1-\lambda )\lambda ^{N}))}{\mu (\Delta
_{j}(x)\cup \Delta_{j}^{R}\cup \Delta_{j}^{L})}\frac{\mu (\Delta _{j}(x)\cup
\Delta_{j}^{R}\cup \Delta_{j}^{L})}{\mu (B(x,3\lambda ^{n}))} \\
&\leq &ct^{j+\kappa_{j+1}+1-N}q^{j}t^{n-j}\leq c\,q^{n}t^{n-N+\kappa
_{n}}\leq c2^{n-N}t^{n-N+\kappa_{n}}.
\end{eqnarray*}

Since $(\log f(N+1))/N$ and $\kappa_n/n$ tend to zero for increasing $n,N$,
there exists $N_1$ such that 
\begin{equation*}
\left( 1-\frac{N}{n} \right)\log(2t\lambda^{2\beta}) \leq 2\beta\frac{\log
f(N+1)}{n}-\frac{\kappa_n}{n}\log t
\end{equation*}
for all $n\geq(1+\delta)N\geq N_1$. Thus 
\begin{equation*}
(2t)^{n-N}t^{\kappa_{n}}\leq \left( f(N+1)\lambda ^{-(n-N)}\right) ^{2\beta}
\end{equation*}%
and that ensures 
\begin{equation}
\frac{\mu (B(x,f(N)(1-\lambda )\lambda ^{N}))}{\mu (B(x,3\lambda ^{n}))}\leq
c\left( \frac{f(N+1)(1-\lambda )\lambda ^{N}}{3\lambda ^{n}}\right) ^{2\beta}
\label{3}
\end{equation}%
for all $N\geq N_1$.

\vskip1em

Having established the same upper bound in \eqref{1}, \eqref{2}, and %
\eqref{3} it remains to show that this is sufficient for the quasi-Assouad
dimension to be finite. Let $N_2=\max\{N_0,N_1\}$. For fixed $\delta^{\prime
}>\delta$, let $r\leq R^{1+\delta^{\prime }}\leq R <
f(N_2)(1-\lambda)\lambda^{N_2}$ and choose $N\geq N_2$ and $n$ such that 
\begin{equation*}
f(N+1)(1-\lambda )\lambda ^{N+1}\leq R\leq f(N)(1-\lambda )\lambda ^{N}
\qquad\text{and}\qquad 3\lambda ^{n}\leq r\leq 3\lambda ^{n-1}.
\end{equation*}
We note that this is well-defined as $f$ is non-increasing and $r\leq
R^{1+\delta^{\prime }}$ gives $n\geq (1+\delta)N$. Now, by appealing to %
\eqref{1}, \eqref{2}, and \eqref{3}, we have%
\begin{equation*}
\frac{\mu (B(x,R))}{\mu (B(x,r))}\leq \frac{\mu (B(x,f(N)(1-\lambda )\lambda
^{N}))}{\mu (B(x,3\lambda ^{n}))}\leq c\left( \frac{f(N+1)(1-\lambda
)\lambda ^{N}}{3\lambda ^{n}}\right) ^{2\beta}\leq c\left( \frac{R}{r}%
\right) ^{2\beta}.
\end{equation*}%
This proves that $\dim\sb{\,\mathrm{qA}}\,\mu \leq 2\beta <\infty$.
\end{proof}

We finish this section by providing a classification of the quasi-Assouad
dimension in terms of quasi-doubling.

\begin{corollary}
Let $\mu $ be a self-similar measure that satisfies the asymptotic gap weak
separation condition and with $\supp\mu =[0,1]$. Then $\dim %
\sb{\,\mathrm{qA}}\,\mu <\infty $ if and only if $\mu $ is quasi-doubling.
\end{corollary}

\begin{proof}
The first part of the proof of Theorem \ref{qdimpliesqA} actually shows that
quasi-doubling implies quasi-net doubling and thus has finite quasi-Assouad
dimension.
\end{proof}

\begin{corollary}
\label{char} Let $\mu$ be a self-similar measure with $\supp\mu=[0,1]$.
Suppose $\mu $ satisfies the weak separation condition with minimal
contraction factor $\lambda $. Then $\dim \sb{\,\mathrm{qA}}\,\mu <\infty $
if and only if for every $q>1$ and $0<A<B$ there exist constants $c=c(q,A,B)$
such that%
\begin{equation}
\mu (B(x,A\lambda ^{n}))\geq c\,q^{-n}\mu (B(x,B\lambda ^{n}))
\label{wscdoubling}
\end{equation}%
for all $n\in \mathbb{N}$ and $x\in \supp\mu $. Equivalently, $\dim %
\sb{\,\mathrm{qA}}\,\mu <\infty $ if and only if for every $\varepsilon >0,$
there is a constant $C$ such that 
\begin{equation*}
\mu (B(x,r))\geq Cr^{\varepsilon }\mu (B(x,2r))
\end{equation*}
for all $r\in (0,1)$ and $x\in \supp\mu$.
\end{corollary}

Note that this Corollary includes Theorem A.

\begin{proof}
As we observe in (\ref{qndwsc}), satisfying (\ref{wscdoubling}) with $%
A=a\lambda /2$ for $a$ the constant arising in the definition of the WSC and 
$B=3$ is enough to ensure the measure is quasi-net doubling and hence has
finite Assouad dimension.
\end{proof}

\section{Dimensions of weakly comparable and generalized regular measures}

\label{sect:comparable} The equicontractive self-similar measure $\mu $
arising from the weighted IFS, $\{S_{j},p_{j}\},$ \ is said to be \textit{%
regular} if $p_{0}=p_{m-1}=\min p_{j}$. In this section we study the more
general classes of \emph{generalized regular} and \emph{weakly comparable
measures} with the goal of proving Theorem B.

\subsection{Weakly comparable measures\label{sect:wc}}

\begin{definition}
A self-similar measure $\mu $ with support $[0,1]$ is \textbf{weakly
comparable} if for each $q>1$ there is a constant $c$, depending on $q$,
such that for all $n$ and adjacent net intervals $\Delta _{1},\Delta _{2}$
of level $n$, we have%
\begin{equation}
\frac{1}{c}q^{-n}P_{n}(\Delta _{2})\leq P_{n}(\Delta _{1})\leq
c\,q^{n}P_{n}(\Delta _{2}).  \label{comp}
\end{equation}
\end{definition}

Weakly comparable measures were originally introduced in \cite{HHS} where
they were called comparable measures. Property (\ref{comp}) is very useful
for studying dimensional properties since it allows one to approximate the
(often poorly understood) quantities $\mu (B(x,\lambda ^{n}))$ by the (often
better understood) $P_{n}(\Delta )$ for $\Delta $ containing $x$.

In Corollary \ref{RegimpliesComp} we will show that any regular self-similar
measure is weakly comparable. If an IFS is of finite type, then the net
intervals of level $n$ are comparable in size, thus if a finite type measure
is doubling, then it is weakly comparable. On the other hand, Example \ref%
{NotDoubling} gives an equicontractive, finite type IFS that is regular,
hence weakly comparable, but not doubling and therefore has infinite Assouad
dimension.

However, weakly comparable measures have finite quasi-Assouad dimension, as
we see next.

\begin{proposition}
\label{wkcompfinite}If $\mu$ is a weakly comparable, self-similar measure
that satisfies the asymptotic gap weak separation condition, then $\mu $ is
quasi-net doubling and hence has finite quasi-Assouad dimension.
\end{proposition}

\begin{proof}
Fix $q_{0}>1$ and let $q=(q_{0})^{1/3}>1$. By Lemma \ref{PnRel} there is a
constant $A=\min p_j^\Theta$ such that $P_{n}(\Delta _{n}(x))\geq
A\,P_{n-1}(\Delta_{n-1}(x))$ for all $n$ and $x\in[0,1]$. Choose $%
\varepsilon >0$ so that $A^{\varepsilon }\geq q^{-1}$. Given any net
interval $\Delta_{n}(x)$ we have 
\begin{equation*}
\mu (\Delta_{n}^{\ast })\leq P_{n}(\Delta_{n}^{L})+P_{n}(\Delta
_{n}^{R})\leq c\,q^{2n}P_{n}(\Delta_{n}(x)),
\end{equation*}
where $\Delta^\ast_n, \Delta^L_n,\Delta^R_n$ are as in Definition \ref%
{def:netDoubl}. Note that by the definition of quasi-net doubling we only
need to check the case when $l(\Delta_{n}(x))\geq f(n+1)\lambda ^{n+1}$.
Taking $z$ to be the midpoint of $\Delta_{n}(x)$ we have 
\begin{equation*}
\Delta_{n}(x)\supseteq B(z,f(n+1)\lambda ^{n+1}/2)\supseteq
B(z,f(n+2)\lambda ^{n+2}),
\end{equation*}%
hence Lemma \ref{PnRel} and \eqref{PnUpper} yield 
\begin{equation*}
\mu (\Delta_{n}(x))\geq \mu (B(z,f(n+2)\lambda ^{n+2}))\geq P_{n+2+\kappa
_{n+2}}(\Delta_{n+2+\kappa_{n+2}}(x))\geq A^{2+\kappa_{n+2}}P_{n}(\Delta
_{n}(x)).
\end{equation*}

For large enough $n$, the weakly comparable assumption thus implies%
\begin{eqnarray*}
\mu (\Delta_{n}(x)) &\geq &A^{2}A^{n\varepsilon }P_{n}(\Delta_{n}(x))\geq
A^{2}q^{-n}P_{n}(\Delta_{n}(x)) \\
&\geq &A^{2}q^{-3n}c^{-1}\mu (\Delta_{n}^{\ast })\geq
A^{2}q_{0}^{-n}c^{-1}\mu (\Delta_{n}^{\ast }).
\end{eqnarray*}

The inequality $\mu (\Delta_{n}^{\ast })\geq c_2\,q_{0}^{-n}\mu
(\Delta_{n}(x))$ follows analogously. This shows that $\mu$ is quasi-net
doubling.
\end{proof}

The next result is similar in spirit, but more technical, and will be used
later to find upper bounds on the quasi-Assouad dimension.

\begin{lemma}
Suppose $\mu$ is a weakly comparable, self-similar measure that satisfies
the asymptotic gap weak separation condition with function $f(n)$. Then for
any $q>1$ there are constants $c_{1},c_{2}>0$ depending on $q$ such that,
for all $x\in[0,1]$ and $n\in\N$, 
\begin{equation}
c_{1}q^{-n}P_{n}(\Delta_{n}(x))\leq P_{n+\kappa_{n}}(\Delta_{n+\kappa
_{n}}(x))\leq \mu (B(x,f(n)\lambda ^{n}))\leq c_{2}q^{n}P_{n}(\Delta
_{n}(x)).  \label{Comp}
\end{equation}
\end{lemma}

\begin{proof}
Since $l(\Delta_{n}^{\ast})\geq f(n)\lambda ^{n}$, we have 
\begin{equation*}
B(x,f(n)\lambda ^{n})\cap \lbrack 0,1]\subseteq \Delta_{n}(x)\cup \Delta
_{n}^{R}\cup \Delta_{n}^{L},
\end{equation*}
so that $\mu (B(x,f(n)\lambda ^{n}))\leq c q^{n}P_{n}(\Delta_{n}(x))$ for
some $c>0$. Similar reasoning to the above shows that $P_{n+\kappa_{n}}(%
\Delta _{n+\kappa_{n}}(x))\geq q^{-n}P_{n}(\Delta_{n}(x))$ for $n$
sufficiently large and as we always have $P_{n+\kappa_{n}}(\Delta_{n+\kappa
_{n}}(x))\leq \mu (B(x,f(n)\lambda ^{n}))$, the inequalities of (\ref{Comp})
are complete.
\end{proof}

\subsection{Generalized regular measures}

\label{sect:regular}

To define generalized regular measures, we first need to introduce further
terminology. We assume $\supp\mu =[0,1]$.

Suppose $\Delta \in \mathcal{F}_{N}$ has descendent net subinterval $\Delta
^{\prime }\in \mathcal{F}_{N+n}$. If $u\in \Lambda _{N}$ with $%
S_{u}[0,1]\supseteq \Delta $, then there is some word $w$ such that $uw\in
\Lambda _{n+N}$ and $S_{uw}[0,1]\supseteq \Delta ^{\prime }$. We call such a
word $w$ a \textit{path of level }$n$ (of $\Delta $). Clearly, 
\begin{equation*}
P_{N+n}(\Delta _{N+n}(x))\geq \inf \{p_{w}:w\text{ path of level }%
n\}P_{N}(\Delta _{N}(x)).
\end{equation*}%
We call $w$ a \textit{left-edge path} if $S_{w}(0)=0$ and a \textit{%
right-edge path} if $S_{w}(1)=1$. Put 
\begin{equation*}
\Gamma _{\Delta ,n}^{L}=\sum_{\substack{ w\text{ left-edge path}  \\ \text{%
of }\Delta \text{ of level }n}}\hspace{-1.5em}p_{w}\qquad \text{ and }\qquad
\Gamma _{n}^{L}=\Gamma _{\lbrack 0,1],n}^{L}.
\end{equation*}%
We define $\Gamma _{\Delta ,n}^{R}$ and $\Gamma _{n}^{R}$ similarly, and set 
\begin{equation*}
\Gamma _{n}=\Gamma _{n}^{L}+\Gamma _{n}^{R}\text{ and }\Gamma _{\Delta
,n}=\Gamma _{\Delta ,n}^{R}+\Gamma _{\Delta ,n}^{L}.
\end{equation*}%
Note that 
\begin{equation*}
\Gamma _{n}^{L}=\sum_{\substack{ w\in \Lambda _{n}  \\ S_{w}(0)=0}}%
p_{w}=P_{n}(\Delta _{n}(0))
\end{equation*}%
hence the constants $s,t$ introduced in (\ref{s,t}) are also equal to%
\begin{equation*}
s=\liminf_{n\rightarrow \infty }\left( \Gamma _{n}^{L}\right) ^{1/n}\text{, }%
t=\liminf_{n\rightarrow \infty }\left( \Gamma _{n}^{R}\right) ^{1/n}.
\end{equation*}

For each positive integer $n$ and $x\in \lbrack 0,1]$, let 
\begin{equation*}
Q_{n}(x)=\sup_{N\in \N}\frac{P_{N}(\Delta _{N}(x))}{P_{N+n}(\Delta _{N+n}(x))%
}
\end{equation*}%
where $\Delta _{N+n}(x)$ is a child of $\Delta _{N}(x)$ containing $x$ (with 
$P_{N+n}(\Delta _{N+n}(x))$ minimal if there are two choices). Set%
\begin{equation*}
Q_{n}=\sup_{x\in \supp\mu }Q_{n}(x).
\end{equation*}%
Of course, $Q_{n}\leq Q_{1}^{n}$ and by Lemma \ref{PnRel}, $Q_{1}\leq
(\min_{j}p_{j})^{-\Theta }$ where $\Theta $ is given by %
\eqref{eq:ThetaDefinition}.

\begin{definition}
\label{genreg}The weighted iterated function system $\{S_{j},p_{j}\}$ is 
\textbf{generalized regular} if for each $q>1$, 
\begin{equation*}
\lim_{n\rightarrow \infty }Q_{n}q^{-n}\sup_{\Delta }\Gamma _{\Delta ,n}=0,
\end{equation*}%
where the supremum is taken over all net intervals.

We will also call the self-similar measure $\mu $ associated with $%
\{S_{j},p_{j}\}$ a generalized regular measure.
\end{definition}

In order to show that regular measures are generalized regular, we first
prove that all $\Gamma _{\Delta ,n}$ are comparable to $\Gamma _{n}.$

\begin{proposition}
There exists $c>0$ such that $\Gamma_n^{L}\leq \Gamma_{\Delta ,n}^{L}\leq
c\Gamma _{n}^{L}$ for all $n$ and $\Delta$. Similarly, $\Gamma_{n}^{R}$ is
comparable to $\Gamma_{\Delta,n}^{R}$.
\end{proposition}

\begin{proof}
Let $\Delta =[a,b]$ be a net interval of level $N$. Then $\Gamma_{\Delta
,n}^{L}=\sum p_{w}$, where the sum is over all $w$ where $S_{w}(0)=0$ and
there is some $u \in \Lambda_{N}$ such that $uw\in \Lambda_{N+n}$ and $%
S_{u}(0)=a$. This means that $w,u$ must satisfy the conditions $r_{u}\leq
\lambda ^{N}$, $r_{u^{-}}>\lambda ^{N}$, $r_{uw}\leq \lambda ^{N+n}$ and $%
r_{(uw )^{-}}=r_{uw^{-}}>\lambda ^{N+n}$. Since $r_{j}\geq \lambda$ for all $%
j$, it follows that $r_{w}<\lambda ^{n-1}$ and $r_{w ^{-}}>\lambda ^{n}$.
Consequently,%
\begin{eqnarray}
\Gamma_{\Delta ,n}^{L} &\leq &\hspace{-1.3em}\sum_{\substack{ S_{w}(0)=0, 
\\ r_{w}<\lambda ^{n-1},r_{w^{-}}>\lambda ^{n}}}\hspace{-1.5em}p_{w}  \notag
\\
&=&\hspace{-1em}\sum_{\substack{ S_{w}(0)=0,  \\ r_{w}\leq \lambda
^{n},r_{w^{-}}>\lambda ^{n}}}\hspace{-1.5em}p_{w}\hspace{1.5em}+\hspace{-1em}%
\sum_{\substack{ S_{w }(0)=0,  \\ r_{w}\in (\lambda ^{n},\lambda
^{n-1}),r_{w ^{-}}>\lambda ^{n}}}\hspace{-1.5em}p_{w}.  \label{eq:twoSums}
\end{eqnarray}

Recall that $w\in \Lambda_{n}$ is equivalent to $w$ satisfying $r_{w}\leq
\lambda ^{n}$ and $r_{w^{-}}>\lambda ^{n}$, thus the left sum in %
\eqref{eq:twoSums} is equal to $\Gamma_{n}^{L}$. For the second sum, let $k$
be the minimal integer such that $r_{0}^{k}\leq \lambda .$ For fixed $%
r_{w}\in (\lambda ^{n},\lambda ^{n-1})$, choose $j=j(w)\in \{1,2,...,k\}$
such that $r_{w }r_{0}^{j}\leq \lambda ^{n}$ and $r_{w}r_{0}^{j-1}>\lambda
^{n}$. Let $\overline{0}_j$ be the unique word of length $j$ containing just
the letter $0$. Then%
\begin{eqnarray*}
\sum_{\substack{ S_{w }(0)=0,  \\ r_{w}\in (\lambda ^{n},\lambda
^{n-1}),r_{w ^{-}}>\lambda ^{n}}}\hspace{-2.5em}p_{w} &\leq
&\sum_{j=1}^{k}\;\sum_{\substack{ S_{w}(0)=0  \\ w \overline{0}_{j}\in
\Lambda_{n}}}p_{w }\;\;=\;\;\;\sum_{j=1}^{k}p_{0}^{-j}\sum_{\substack{ %
S_{\tau}(0)=0  \\ \tau=w \overline{0}_{j}\in \Lambda_{n}}}p_{\tau} \\
&\leq &\sum_{j=1}^{k}p_{0}^{-j}\sum_{\substack{ S_{\tau}(0)=0  \\ \tau\in
\Lambda_{n}}}p_{\tau}\leq kp_{0}^{-k}\Gamma_{n}^{L}\text{.}
\end{eqnarray*}

We conclude $\Gamma _{\Delta ,n}^{L}\leq (1+kp_{0}^{-k})\Gamma _{n}^{L}$, as
required.
\end{proof}

\begin{corollary}
\label{RegimpliesComp}(i) A self-similar measure $\mu $ is generalized
regular if and only if 
\begin{equation*}
\lim_{n}Q_{n}q^{-n}\Gamma _{n}^{L}=\lim_{n}Q_{n}q^{-n}\Gamma _{n}^{R}=0\quad 
\text{ for every }\quad q>1\text{.}
\end{equation*}

(ii) An equicontractive, regular self-similar measure ($p_{0}=p_{m-1}=%
\min_{j}p_{j}$) is generalized regular.
\end{corollary}

\begin{proof}
(i) follows immediately from the Proposition.

(ii) The edge paths of level $n$ of such a measure are the words $(0)^{n}$
and $(1)^{n}$, so $\Gamma _{n}=p_{0}^{n}+p_{m-1}^{n}$. Furthermore, $%
Q_{n}\leq (\min p_{j})^{-n}=p_{0}^{-n}=p_{m-1}^{-n}$, hence the measure is
generalized regular.
\end{proof}

The notion of `generalized regular' was introduced in the study of
non-equicontractive finite type iterated function systems where it was
observed that generalized regular implies weakly comparable (see \cite[%
Theorem 4.11]{HHS}). In fact, this holds in general.

\begin{proposition}
A generalized regular, self-similar measure $\mu $ is weakly comparable.
\end{proposition}

\begin{proof}
Fix $q>1$ and choose $N_{0}$ so $\sup_{\Delta }\Gamma_{\Delta ,n}\leq
Q_{n}^{-1}q^{n}/2$ for all $n\geq N_{0}$. Since $P_{n}(\Delta)$ is finite
for all $n$, we can find $c>0$ such that%
\begin{equation*}
\frac{1}{c}q^{-k}P_{k}(\Delta_{2})\;\leq\; P_{k}(\Delta_{1})\;\leq\;
c\,q^{k}P_{k}(\Delta_{2})
\end{equation*}%
whenever $\Delta_{1},\Delta_{2}$ are adjacent net intervals of level $k$ for
all $k=1,...,N_{0}$.

Assume $n\geq N_{0}+1$. We proceed by induction on $n$. Suppose $\Delta
_{1},\Delta_{2}$ are adjacent net intervals of level $n$ where, without loss
of generality, $\Delta_{1}$ is to the left of $\Delta_{2}$. If $\widehat{%
\Delta}_{j}$ is the ancestor of $\Delta_{j}$ at level $n-k$, then $%
P_{n}(\Delta_{j})\sim P_{n-k}(\widehat{\Delta}_{j})$, with constants of
comparability depending only on $k$. Thus we can assume $\Delta_{1},\Delta
_{2}$ have no common ancestor within $N_{0}$ levels.

For $j=1,2$, let $\widehat{\Delta_{j}}$ be the $(n-N_{0})$-level ancestor of 
$\Delta_{j}$. Let $\mathcal{D}_{1}$ denote the words $u\in \Lambda
_{n-N_{0}} $ where $S_{u}[0,1]$ contains $\widehat{\Delta_{1}}$, but not $%
\widehat{\Delta_{2}}.$ Define $\mathcal{D}_{2}$ analogously and let $%
\mathcal{E}$ denote those $u\in \Lambda_{n-N_{0}}$ where $S_{u }[0,1]$
contains both $\widehat{\Delta_{1}}$ and $\widehat{\Delta_{2}}$.

Consider any $\tau\in \Lambda_{n}$ with $S_{\tau}[0,1]$ covering $\Delta
_{1} $. Then $\tau=u w$ where $u\in \Lambda_{n-N_{0}}$, $S_{u}[0,1]$
contains $\widehat{\Delta_{1}}$ and $w$ is a path of level $N_{0}$ of $%
\widehat{\Delta_{1}}$. The word $u$ belongs to either $\mathcal{D}_{1}$ or $%
\mathcal{E}$ and in the former case $w$ is a right edge path. Thus%
\begin{eqnarray*}
P_{n}(\Delta_{1}) &=&\hspace{-2em}\sum_{\substack{ u\in \Lambda_{n-N_{0}} 
\\ w \text{ path of level }N_{0}\text{ of }\widehat{\Delta_{1}}  \\ S_{uw
}[0,1]\supseteq \Delta_{1}}}\hspace{-2.5em}p_{u w} \\
\\
&=&\hspace{-3em}\sum_{\substack{ u \in \mathcal{D}_{1}  \\ w\text{ right
edge path of level }N_{0}  \\ S_{uw}[0,1]\supseteq \Delta_{1}}}\hspace{-3.5em%
}p_{u }p_{w}\hspace{2.0em}+\sum_{\substack{ u\in \mathcal{E}  \\ w\text{
path of level }N_{0}  \\ S_{u w}[0,1]\supseteq \Delta_{1}}}\hspace{-2em}%
p_{u}p_{w} \\
&\leq &\sum_{u \in \mathcal{D}_{1}}p_{u}\Gamma_{\widehat{\Delta _{1}}%
,N_{0}}^{R}+\sum_{u\in \mathcal{E}}p_{u} \\
&\leq &\frac{Q_{N_{0}}^{-1}q^{N_{0}}}{2}\sum_{u\in \mathcal{D}%
_{1}}p_{u}+\sum_{u\in \mathcal{E}}p_{u}.
\end{eqnarray*}%
Now $\sum_{u\in \mathcal{D}_{1}}p_{u}\leq \sum_{u \in \mathcal{D}%
_{1}}p_{u}+\sum_{u\in \mathcal{E}}p_{u }=P_{n-N_{0}}(\widehat{\Delta_{1}})$
and $\sum_{u \in \mathcal{E}}p_{u}\leq \sum_{u\in \mathcal{E}}p_{u}+\sum_{u
\in \mathcal{D}_{2}}p_{u}=P_{n-N_{0}}(\widehat{\Delta_{2}})$. Hence applying
the inductive assumption and using the fact that $P_{n-N_{0}}(\widehat{%
\Delta_{2}})\leq P_{n}(\Delta_{2})Q_{N_{0}}$ we have%
\begin{eqnarray*}
P_{n}(\Delta_{1}) &\leq &\frac{Q_{N_{0}}^{-1}q^{N_{0}}}{2}P_{n-N_{0}}(%
\widehat{\Delta_{1}})+P_{n-N_{0}}(\widehat{\Delta_{2}}) \\
&\leq &\left( \frac{Q_{N_{0}}^{-1}q^{N_{0}}}{2}c\,q^{n-N_{0}}+1\right)
P_{n-N_{0}}(\widehat{\Delta_{2}}) \\
&\leq &\left( \frac{1}{2}Q_{N_{0}}^{-1}c\,q^{n}+1\right) P_{n}(\Delta
_{2})Q_{N_{0}}.
\end{eqnarray*}%
Taking $c\geq 0$ sufficiently large, we obtain the desired conclusion that $%
P_{n}(\Delta_{1})\leq c\,q^{n}P_{n}(\Delta_{2})$.
\end{proof}

\subsection{Dimensions of generalized regular measures\label{sect:dimreg}}

Theorem~B is a consequence of Theorem \ref{thm:theLocDim} below.{\ Other
examples of measures satisfying the assumptions of the Theorem below are
Bernoulli convolutions with contraction ratios being reciprocals of Salem
numbers.}

\begin{theorem}
\label{thm:theLocDim} \label{qA=maxLocal}Suppose $\mu $ is a generalized
regular, self-similar measure, with support $[0,1]$, that satisfies the
asymptotic gap weak separation condition. Then 
\begin{equation*}
\dim \sb{\,\mathrm{qA}}\,\mu =\max \left\{ \overline{\dim }_{\loc}\,\mu (0),%
\overline{\dim }_{\loc}\,\mu (1)\right\} =\max \{\overline{\dim }_{\loc%
}\,\mu (x):x\in \supp\mu \}.
\end{equation*}
\end{theorem}

\begin{corollary}
\label{cor:regqA}If $\mu $ is a regular, equicontractive self-similar
measure with full support and satisfying the weak separation condition, then 
\begin{equation*}
\dim \sb{\,\mathrm{qA}}\,\mu =\max \{\overline{\dim }_{\loc}\,\mu (x):x\in %
\supp\mu \}.
\end{equation*}
\end{corollary}

\begin{remark}
In Example \ref{notfull} we see this can fail if the measure does not have
full support.
\end{remark}

\begin{proof}[Proof of Theorem \ref{thm:theLocDim}]
Without loss of generality, assume $\max \left\{ \overline{\dim }_{\loc%
}\,\mu (0),\overline{\dim }_{\loc}\,\mu (1)\right\} =\overline{\dim }_{\loc%
}\,\mu (0)$, which by Proposition \ref{LocDim0} is equal to 
\begin{equation*}
d=\log s/\log \lambda \text{ where }s=\liminf \left( \Gamma _{n}^{L}\right)
^{1/n}.
\end{equation*}

First, we will verify that for small enough $\varepsilon >0$ and every $%
\delta >0$ there are constants $C,C_{0}$, depending on $\varepsilon $, $%
\delta $, such that if $r\leq R^{1+\delta }\leq R\leq C_{0}$, then for all $%
x $,%
\begin{equation*}
\frac{\mu (B(x,R))}{\mu (B(x,r))}\leq C\left( \frac{R}{r}\right)
^{d(1+\varepsilon )}.
\end{equation*}%
Consequently, $\dim \sb{\,\mathrm{qA}}\,\mu \leq d$.

Fix $\delta ,\varepsilon >0$ and $x\in \lbrack 0,1]$. Assume $r\leq
R^{1+\delta }\leq R\leq C_{0}$ where $C_{0}$ will be specified later. Choose
integers $N,n$ such that $f(N+1)\lambda ^{N+1}<R\leq f(N)\lambda ^{N}$ and $%
\lambda ^{n}\leq r<\lambda ^{n-1}$. By (\ref{PnLower}), 
\begin{equation*}
\mu (B(x,r))\geq \mu (B(x,\lambda ^{n}))\geq P_{n}(\Delta _{n}(x)).
\end{equation*}%
Since generalized regular measures are weakly comparable, 
\begin{equation*}
\mu (B(x,R))\leq \mu (B(x,f(N)\lambda ^{N}))\leq \mathcal{C}%
q^{N}P_{N}(\Delta _{N}(x))
\end{equation*}%
for $\mathcal{C}$ depending on $q$, as per \eqref{Comp}.

We see that 
\begin{equation*}
\frac{\mu (B(x,R))}{\mu (B(x,r))}\leq \frac{\mathcal{C}q^{N}P_{N}(\Delta
_{N}(x))}{P_{n}(\Delta _{n}(x))}\leq \mathcal{C}q^{N}Q_{n-N}.
\end{equation*}%
As $\mu $ is generalized regular, $\Gamma _{m}^{L}Q_{m}q^{-m}\rightarrow 0$
as $m\rightarrow \infty $. Thus, we can choose $N_{1}$ such that if $%
m=n-N\geq \delta N_{1}$ then $Q_{m}\leq q^{m}(\Gamma _{m}^{L})^{-1}$.
Therefore 
\begin{equation*}
\frac{\mu (B(x,R))}{\mu (B(x,r))}\leq \frac{q^{m}\mathcal{C}q^{N}}{\Gamma
_{m}^{L}}=\mathcal{C}q^{n}(\Gamma _{m}^{L})^{-1}.
\end{equation*}

As noted in Proposition \ref{LocDim0}, $s=\liminf \left( \Gamma
_{n}^{L}\right) ^{1/n}\in (0,1)$. Thus we can choose $N_{2}$ so that if $%
m\geq \delta N$ for some $N\geq N_{2}$, then 
\begin{equation*}
\Gamma _{m}^{L}\geq s^{m(1+\varepsilon /2)}.
\end{equation*}%
We require that $C_{0}$ be so small that if $R\leq C_{0}$, then $R\leq
f(N)\lambda ^{N}$ for $N\geq \max (N_{1},N_{2})$. Hence for $r\leq
R^{1+\delta }\leq R\leq C_{0}$ we have%
\begin{equation*}
\frac{\mu (B(x,R))}{\mu (B(x,r))}\leq \mathcal{C}q^{n}s^{-m(1+\varepsilon
/2)}.
\end{equation*}%
As $s=\lambda ^{d}$, for any fixed $\varepsilon >0$, there exists $c>0$ such
that 
\begin{equation}
\left( \frac{R}{r}\right) ^{d(1+\varepsilon )}\geq \lambda
^{(N-n+2)d(1+\varepsilon )}f(N+1)^{d(1+\varepsilon
)}=cf(N+1)^{d(1+\varepsilon )}s^{-m(1+\varepsilon )}.  \label{R/r}
\end{equation}%
We note that 
\begin{equation*}
\mathcal{C}q^{n}s^{-m(1+\varepsilon /2)}\leq cf(N+1)^{d(1+\varepsilon
)}s^{-m(1+\varepsilon )}
\end{equation*}%
if and only if 
\begin{equation*}
C^{\prime }q^{n}f(N+1)^{-d(1+\varepsilon )}\leq s^{-m\varepsilon /2},
\end{equation*}%
where $C^{\prime }$ is the appropriate constant. Taking logarithms, this is
equivalent to 
\begin{equation*}
\frac{1}{m}\left( \log C^{\prime }\text{ }+n\log q-d(1+\varepsilon )\log
f(N+1)\right) \leq \varepsilon \left\vert \log s\right\vert /2.
\end{equation*}%
Now, $m=n-N\geq \delta n/(1+\delta )$, so $(n\log q)/m\leq (1+\delta )(\log
q)/\delta $ and $(1/m)\log f(N+1)\rightarrow 0$. Thus with a suitable choice
of $q$ close to $1$ (depending on $\varepsilon ,\delta )$ and large enough $%
N $, we can achieve this inequality. With this further constraint on $C_{0}$
it then follows that for a suitable constant $c$, we have%
\begin{equation*}
\frac{\mu (B(x,R))}{\mu (B(x,r))}\leq c\left( \frac{R}{r}\right)
^{d(1+\varepsilon )}\text{ for all }r\leq R^{1+\delta }\leq R\leq C_{0},
\end{equation*}%
and this implies $\dim \sb{\,\mathrm{qA}}\,\mu \leq d(1+\varepsilon )$ for
all $\varepsilon >0$. Hence $\dim \sb{\,\mathrm{qA}}\,\mu \leq d$

By Proposition \ref{prop:boundedBylocDim} we have $\dim \sb{\,\mathrm{qA}}%
\,\mu \geq \sup \{\overline{\dim }_{\loc}\,\mu (x):x\in \supp\mu \}\geq d,$
giving the lower bound and hence equality.
\end{proof}

In contrast, generalized regular measures satisfying the AGWSC and having
full support need not have finite Assouad dimension.

\begin{example}
\label{NotDoubling} Consider the IFS $S_{j}(x)=x/3+d_{j}$ where $d_{0}=0$, $%
d_{1}=1/6$, $d_{2}=1/3$, $d_{3}=2/3$ and probabilities each $1/4$. This IFS
is equicontractive, finite type, regular and of full support. Thus it is
generalized regular and hence weakly comparable. Applying Theorem \ref%
{qA=maxLocal} gives $\dim _{qA}\mu =\dim _{loc}\mu (0)=\log 4/\log 3$.

However, $\mu $ is not doubling and consequently, the Assouad dimension of $%
\mu $ is infinite. To see this, one can show that $1/2$ is the boundary
point of two level $n$ net intervals for each level $n$. The two intervals
have the same length, $3^{-n}/2$. Using the techniques developed in \cite%
{HHN} it can be shown that the $\mu $-measure of the right interval is at
most $c_{1}4^{-n}$ for some constant $c_{1}>0$, while the left interval has
measure at least $c_{2}n4^{-n}$ for some $c_{2}>0$. With $R=(\frac{3}{4}%
)3^{-n}$, $r=(\frac{1}{4})3^{-n}$ and $x_{n}$ the midpoint of the right net
interval of level $n$, we have $\mu (B(x_{n},R))\geq c_{2}n4^{-n}$ and $\mu
(B(x_{n},r))\leq c_{1}4^{-n}$, while $R/r=3,$ which proves $\mu $ is not
doubling.
\end{example}

\subsection{Dimensions of weakly comparable measures\label{sect:dimwc}}

For the larger class of weakly comparable measures we can also obtain bounds
on the quasi-Assouad dimension. These bounds wil be used to show in Example %
\ref{wkcomp} that the generalized regular condition in Theorem~\ref%
{thm:theLocDim} is not necessary.

\begin{theorem}
\label{LocComp}Suppose $\mu$ is weakly comparable and satisfies the
asymptotic gap weak separation property. Then the quasi-Assouad dimension of 
$\mu$ is bounded by 
\begin{equation}
\dim\sb{\,\mathrm{qA}}\,\mu \leq \limsup_{n\to\infty}\frac{-\log Q_{n}}{n
\log \lambda }<\infty ,  \label{qdimQ}
\end{equation}%
and for all $x\in \lbrack 0,1]$, 
\begin{eqnarray*}
\overline{\dim}_{\loc}\,\mu (x) &=&\limsup_{n\to\infty}\frac{\log
P_{n}(\Delta_{n}(x))}{n\log \lambda } \\
\underline{\dim}_{\loc}\,\mu (x) &=&\liminf_{n\to\infty}\frac{\log
P_{n}(\Delta_{n}(x))}{n\log \lambda }.
\end{eqnarray*}
\end{theorem}

\begin{proof}
We first prove the upper bound. Let%
\begin{equation*}
d=\limsup_{n}\frac{-\log Q_{n}}{n\log \lambda}
\end{equation*}
It will be enough to prove that $\dim\sb{\,\mathrm{qA}}\,\mu \leq
d+\varepsilon$ for all $\varepsilon >0$. Choose $M$ such that 
\begin{equation}  \label{eq:estimateQ}
(1/m)\log Q_{m}\leq -(d+\varepsilon /2)\log \lambda,
\end{equation}
for all $m\geq M$.

Recall that $f(n)$ is decreasing in $n$. Thus, given $r\leq R^{1+\delta
}\leq R\leq C_{0}$ we can choose $N$ and $n$ such that 
\begin{equation*}
f(N+1)\lambda ^{N+1}<R\leq f(N)\lambda ^{N}
\end{equation*}%
and%
\begin{equation*}
f(n)\lambda ^{n}\leq r<f(n-1)\lambda ^{n-1},
\end{equation*}%
where $C_{0}$ is chosen sufficiently small to ensure that $N+1\leq n-1$ and $%
n-N\geq M$.

Take $q>1$. Then (\ref{Comp}) yields 
\begin{equation*}
\mu (B(x,R))\quad \leq \quad \mu (B(x,f(N)\lambda ^{N}))\quad \leq \quad
Cq^{N}P_{N}(\Delta _{N}(x))
\end{equation*}%
and 
\begin{equation*}
\mu (B(x,r))\quad \geq \quad \mu (B(x,f(n)\lambda ^{n}))\quad \geq \quad
P_{n+\kappa _{n}}(\Delta _{n+\kappa _{n}}(x)),
\end{equation*}%
where $C>0$ depends on $q$. Thus%
\begin{equation*}
\frac{\mu (B(x,R))}{\mu (B(x,r))}\;\leq \;\frac{C\,q^{N}P_{N}(\Delta _{N}(x))%
}{P_{n+\kappa _{n}}(\Delta _{n+\kappa _{n}}(x))}\;\leq \;C\,q^{N}Q_{n+\kappa
_{n}-N}(x).
\end{equation*}%
Using \eqref{eq:estimateQ} and the fact that $\lambda ^{\kappa _{n}}\geq
\lambda f(n)$, we have%
\begin{equation*}
Cq^{N}Q_{n+\kappa _{n}-N}(x)\leq Cq^{N}\lambda ^{-(n+\kappa
_{n}-N)(d+\varepsilon /2)}\leq Cq^{N}\lambda ^{-(n-N)(d+\varepsilon
/2)}f(n)^{-(d+\varepsilon /2)},
\end{equation*}%
redefining $C>0$ as appropriate. Further, 
\begin{equation*}
\left( \frac{R}{r}\right) ^{d+\varepsilon }\geq \left( \frac{f(N+1)}{f(n-1)}%
\right) ^{d+\varepsilon }\lambda ^{(N-n+1)(d+\varepsilon )}\geq c\lambda
^{(N-n)(d+\varepsilon )}.
\end{equation*}%
Thus, in order to satisfy 
\begin{equation*}
\frac{\mu (B(x,R))}{\mu (B(x,r))}\leq c\left( \frac{R}{r}\right)
^{d+\varepsilon }
\end{equation*}%
for some constant $c$, it will be enough to satisfy the inequality%
\begin{equation*}
q^{N}\lambda ^{-(n-N)(d+\varepsilon /2)}f(n)^{-(d+\varepsilon /2)}\leq
\lambda ^{(N-n)(d+\varepsilon )}\text{ }
\end{equation*}%
for all $n\geq (1+\delta )N$, $N$ sufficiently large and for some suitable $%
q>1$. Equivalently, 
\begin{equation*}
(d+\varepsilon /2)\frac{\left\vert \log f(n)\right\vert }{n}\leq \frac{%
(n-N)\varepsilon }{2n}\left\vert \log \lambda \right\vert -\frac{N}{n}\log q%
\text{.}
\end{equation*}

Since $N/n\leq (1+\delta )^{-1}$, the right hand side of the latter
expression dominates 
\begin{equation*}
\left( 1-\frac{1}{1+\delta }\right) \frac{\varepsilon }{2}\left\vert \log
\lambda \right\vert -\frac{1}{1+\delta }\log q,
\end{equation*}%
and this is at least%
\begin{equation*}
\left( 1-\frac{1}{1+\delta }\right) \frac{\varepsilon }{4}\left\vert \log
\lambda \right\vert
\end{equation*}%
if we choose $q$ close enough to $1$. Since $\log f(n)/n\rightarrow 0$ as $%
n\rightarrow \infty$, we can ensure that this quantity dominates $%
(d+\varepsilon /2)\left\vert \log f(n)\right\vert /n$ for large enough $n$.
Suitably redefining $C_{0}>0$, if necessary, will guarantee that $n\geq
(1+\delta )N$ is sufficiently large to be sure this is true. From these
inequalities it follows that the quasi-Assouad dimension of $\mu$ is at most 
$d+\varepsilon$ as required.

\vskip1em We now turn to proving the equalities involving the local
dimension. For each $q>1$ we have, by \eqref{Comp}, 
\begin{equation*}
(Cq^{n})^{-1}P_{n}(\Delta _{n}(x))\leq \mu (B(x,f(n)\lambda ^{n}))\leq
Cq^{n}P_{n}(\Delta _{n}(x))
\end{equation*}%
where $C>0$ depends on $q$. Thus%
\begin{eqnarray*}
\frac{\log q}{\log \lambda }+\liminf_{n\rightarrow \infty }\frac{\log
P_{n}(\Delta _{n}(x))}{n\log \lambda } &\leq &\liminf_{n\rightarrow \infty }%
\frac{\log \mu (B(x,f(n)\lambda ^{n}))}{\log \lambda ^{n}} \\
&\leq &\liminf_{n\rightarrow \infty }\frac{\log P_{n}(\Delta _{n}(x))}{n\log
\lambda }-\frac{\log q}{\log \lambda }.
\end{eqnarray*}%
Since the inequality above holds for all $q>1$, we deduce 
\begin{equation*}
\liminf_{n\rightarrow \infty }\frac{\log \mu (B(x,f(n)\lambda ^{n}))}{\log
\lambda ^{n}}=\liminf_{n\rightarrow \infty }\frac{\log P_{n}(\Delta _{n}(x))%
}{n\log \lambda }
\end{equation*}

Given any $r>0$, choose $n$ such that $f(n+1)\lambda ^{n+1}<r\leq
f(n)\lambda ^{n}$. Since $\log f(n)/n\rightarrow 0$, we deduce 
\begin{align*}
&\liminf_{r\rightarrow 0}\frac{\log \mu (B(x,r)}{\log r}=\liminf_{n\to\infty}%
\frac{\log \mu (B(x,f(n)\lambda ^{n}))}{\log f(n)\lambda^{n}} \\
=\;& \liminf_{n\to\infty}\frac{\log\mu(B(x,f(n) \lambda^n))}{\log \lambda^n}%
=\liminf_{n\to\infty}\frac{\log P_{n}(\Delta_{n}(x))}{n\log \lambda }.
\end{align*}%
Thus $\underline{\dim}_{\loc}\,\mu (x)$ is as claimed.

The arguments for the upper local dimension are identical and left to the
reader.
\end{proof}

\begin{example}
\label{wkcomp}A measure $\mu $ with $\dim \sb{\,\mathrm{qA}}\,\mu =\sup_{x}\{%
\overline{\dim }_{\loc}\mu (x)\}$ that is not generalized regular: Consider
the IFS $S_{0}(x)=x/3$, $S_{1}(x)=x/3+1/3$, $S_{2}(x)=x/3+2/3$, with
probabilities $p_{0}=p_{2}=2/5$, $p_{1}=1/5$ and associated self-similar
measure $\mu $. This is a comparable, but not generalized regular, iterated
function system. Note that $Q_{n}(x)\leq 5^{n}$ for all $x$ and $n$. Thus
Theorem \ref{LocComp} yields that $\dim \sb{\,\mathrm{qA}}\,\mu \leq \log
5/\log 3$, which coincides with $\dim _{\loc}\,\mu (1/2)$, the maximum upper
local dimension. Since the quasi-Assouad dimension is always an upper bound
on the upper local dimension of the measure, we have equality here.
\end{example}

It would be desirable to know if all weakly comparable measures $\mu $ have
the property that $\dim \sb{\,\mathrm{qA}}\,\mu =\sup \{\overline{\dim }_{%
\loc}\mu (x)\;:\;x\in \lbrack 0,1]\}$.

\subsection{An equicontractive, regular, self-similar measure without full
support\label{notfull}}

\mbox{}

\noindent Throughout this paper, we have assumed that the self-similar
measure has the full interval $[0,1]$ as its support. The purpose of the
final example is to show that even an equicontractive, regular measure
without full support need not have finite quasi-Assouad dimension: Consider
the iterated function system with $S_{j}(x)=x/5+d_{j}$ where $d_{0}=0$, $%
d_{1}=1/10$, $d_{2}=2/5$, $d_{3}=4/5$ and probabilities $%
p_{0}=p_{1}=p_{3}=1/6$, $p_{2}=1/2$. This IFS is equicontractive, finite
type and regular, but the self-similar set is clearly not the full interval $%
[0,1]$. Indeed, the subintervals $(3/10,2/5)$ and $(3/5,4/5)$ are in the
complement of the self-similar set.

As explained in \cite{HHN}, we can associate with each net interval of a
finite type IFS a finite tuple called the characteristic vector. The
characteristic vector contains all the information needed to essentially
determine the measure of the net interval given that of its parent net
interval. In fact, finite type is characterized by the property that there
are only finitely many of these so-called characteristic vectors and,
furthermore, each net interval $\Delta _{n}$ of level $n$ can be uniquely
identified by the $(n+1)$-tuple of characteristic vectors, $(\gamma
_{j})_{j=0}^{n},$ where $\gamma _{n}$ is the characteristic vector of $%
\Delta _{n}$, $\gamma _{n-1}$ is the characteristic vector of its parent net
interval, $\Delta _{n-1},$ etc.

This IFS has six characteristic vectors, which we label as $1,2,3a,3b,3c,4$.
Any net interval of level $n-1$ with characteristic vector $3a,3b$, or $3c$
has four children, each of length $5^{-n}/2$. From left to right these are $%
3a,3b,4,3c$, where $4$ and $3c$ are separated by a gap. Fix $\varepsilon >0$
and take any large $N$ and $n=\lceil (1+\delta )N\rceil $. Consider the
level $n$ net interval, $\Delta _{n},$ identified with the tuple, 
\begin{equation*}
(1,\underbrace{3a,3a,...,3a}_{N},\underbrace{3b,3b,...,3b}_{n-N}),
\end{equation*}
and its ancestor $\Delta _{N}$ of level $N$. Let $x_{n}$ denote the midpoint
of $\Delta _{n}$.

Using the techniques of \cite{HHN} it can be shown that%
\begin{equation*}
\mu (\Delta _{N})\sim \left\Vert \lbrack 1/6\text{ }1/6]%
\begin{bmatrix}
1/6 & 0 \\ 
0 & 1/2%
\end{bmatrix}%
^{N}\right\Vert \sim 2^{-N},
\end{equation*}%
while 
\begin{equation*}
\mu (\Delta _{n})\sim \left\Vert \lbrack 1/6\text{ }1/6]%
\begin{bmatrix}
1/6 & 0 \\ 
0 & 1/2%
\end{bmatrix}%
^{N}%
\begin{bmatrix}
1/6 & 1/6 \\ 
0 & 0%
\end{bmatrix}%
^{n-N}\right\Vert \sim 6^{-n},
\end{equation*}%
where the symbol $\sim $ means bounded above and below by some constant
multiple. Taking $R=5^{-N}/2$ and $r=5^{-n}/4$ it follows that 
\begin{equation*}
\frac{\mu (B(x_{n},R))}{\mu (B(x_{n},r))}\geq c_{1}\frac{2^{-N}}{6^{-n}}\geq
c_{1}^{\prime }\frac{6^{(1+\delta )N}}{2^{N}}=c_{1}^{\prime \prime
}3^{N}6^{\delta N},
\end{equation*}%
while $R/r\leq c_{2}5^{\delta N}$. Consequently, $\dim\sb{\,\mathrm{qA}}%
\,\mu =\infty $.


\begin{thebibliography}{99}
\bibitem{assouadphd} P. Assouad. \emph{Espaces m{\'e}triques, plongements,
facteurs}, {\ Th\`ese de doctorat d'\'Etat, Publ. Math. Orsay 223--7769,
Univ. Paris XI, Orsay}, (1977).

\bibitem{F3} D-J. Feng, \textit{Smoothness of the }$L^{q}$\textit{-spectrum
of self-similar measures with overlaps}, J. London Math. Soc. \textbf{68}
(2003), 102--118.

\bibitem{F1} D-J. Feng, \textit{The limited Rademacher functions and
Bernoulli convolutions associated with Pisot numbers}, Adv. in Math., 
\textbf{195 }(2005), 24-101.

\bibitem{FengSalem} D-J. Feng, \textit{Multifractal analysis of Bernoulli
convolutions associated with Salem numbers}, Adv. Math. \textbf{229 }(2012),
3052-3077.

\bibitem{Fraser14} J.~M. Fraser. \emph{Assouad type dimensions and
homogeneity of fractals}, {Trans. Amer. Math. Soc.} \textbf{366 }(2014),
6687--6733.

\bibitem{FHOR} J.~M.~Fraser, A.~M.~Henderson, E.~J.~Olson, and
J.~C.~Robinson, \emph{On the Assouad dimension of self-similar sets with
overlaps}, Adv. Math. \textbf{273 }(2015), 188--214.

\bibitem{FraserHowroyd} J. M. Fraser and D.~Howroyd, \emph{On the upper
regularity dimension of measures}, Indiana Univ. Math. J. (to appear),
arXiv:1706.09340.

\bibitem{FraserTroscheit18} J.~M.~Fraser and S.~Troscheit. \emph{The Assouad
spectrum for random self-affine sets}, Mittag-Leffler preprints \textbf{14801%
} (2018), arXiv:1805.04643.

\bibitem{Hare} I. Garc\'{\i}a and K. E. Hare. \emph{Properties of
Quasi-Assouad dimension}, Camb. Proc. Phil. Soc. (to appear),
arxiv:1703.02526.

\bibitem{HHM} K.E. Hare, K.G. Hare, K.R. Matthews, \textit{Local dimensions
of measures of finite type}, J. Fractal Geometry \textbf{3 }(2016), 331--376.

\bibitem{HHN} K.E. Hare, K.G. Hare and M. K-S. Ng, \textit{Local dimensions
of measures of finite type II - measures without full support and with
non-regular probabilities}, Can. J. Math. \textbf{70 }(2018), 824--867.

\bibitem{HHS} K.E. Hare, K.G. Hare and G. Simms, \textit{Local dimensions of
measures of finite type III - measures that are not equicontractive, }J.
Math. Anal. and Appl. \textbf{458 }(2018), 1653--1677.

\bibitem{Kaen1} A. K\"{a}enm\"{a}ki and J. Lehrb\"{a}ck, \textit{Measures
with predetermined regularity and inhomogeneous self-similar sets}, Ark.
Mat. \textbf{55 }(2017), 165--184.

\bibitem{Kaen2} A. K\"{a}enm\"{a}ki, J. Lehrb\"{a}ck and M. Vuorinen, 
\textit{Dimensions, Whitney covers, and tubular neighborhoods}, Indiana
Univ. Math. J. \textbf{62 }(2013), 1861--1889.

\bibitem{LauNgai} K.-S. Lau and S.-M. Ngai, \textit{Multifractal measures
and a weak separation condition}, Adv.\ Math. \textbf{141 }(1999), 45--96.

\bibitem{Luukkainen} J. Luukkainen. \emph{Assouad dimension: antifractal
metrization, porous sets, and homogeneous measures}, {J. Korean Math. Soc.} 
\textbf{35 }(1998), 23--76.

\bibitem{LuXi} F. L\"{u} and L. Xi. \textit{Quasi-Assouad dimension of
fractals}, {J. Fractal Geom.} \textbf{3 }(2016), 187--215.

\bibitem{NW} S-M. Ngai and Y. Wang, \textit{Hausdorff dimension of
self-similar sets with overlaps}, J. London Math. Soc. \textbf{63 }(2001),
655-672.

\bibitem{Robinson} J. C. Robinson. \emph{Dimensions, Embeddings, and
Attractors}, Cambridge University Press, (2011).

\bibitem{PSS} Y. Peres, W. Schlag and B. Solomyak, \textit{Sixty years of
Bernoulli convolutions}, Fractal geometry and stochastics, II, Progress in
probability \textbf{46}, Birkh\"{a}user, Basel, 2000, 39-65.

\bibitem{Shmerkin14} P. Shmerkin, \textit{On the exceptional set for
absolute continuity of Bernoulli convolutions,} Geom. Func. Anal., \textbf{24%
} (2014), 946--958.

\bibitem{Sh} P. Shmerkin, \textit{A modified multi-fractal formalism for a
class of self-similar measures with overlap}, Asian J. Math., \textbf{9}
(2005), 323-348.

\bibitem{Zerner} M. P. W. Zerner. \emph{Weak separation properties for
self-similar sets}, Proc. Amer. Math. Soc. \textbf{124 }(1996), 3529--3539.
\end{thebibliography}
\end{document}